\documentclass[a4paper, 11pt, twoside, openany]{article}
\usepackage{amsmath,amssymb,stmaryrd,mathtools,amsthm}
\usepackage[all]{xy}
\usepackage{fancyhdr}
\usepackage{footmisc}
\usepackage{hyperref}
\usepackage[margin=22mm]{geometry}
\usepackage[UKenglish]{datetime}
\usepackage{graphicx}
\usepackage{wrapfig}
\usepackage[usenames, dvipsnames]{color}
\usepackage{accents}
\usepackage{cases}
\usepackage{enumerate}
\usepackage{tikz-cd}
\usetikzlibrary{trees,positioning}
\usepackage{multirow}
\usepackage[normalem]{ulem}
\usepackage{algorithm}
\usepackage{algorithmic}

\newcommand{\pp}[2]{\frac{\partial #1}{\partial #2}} 
\newcommand{\dd}[2]{\frac{\delta #1}{\delta #2}}

\newcommand{\bss}[1]{\textcolor{purple}{[BS\@: #1]}}

\newcommand{\KL}[1]{\textcolor{orange}{[KL\@: #1]}}

\newtheorem{definition}{Definition}[section]
\newtheorem{proposition}{Proposition}[section]
\theoremstyle{definition}
\newtheorem{remark}{Remark}[section]

\numberwithin{equation}{section}

\newcounter{savefootnote}
\newcounter{symfootnote}
\newcommand{\symfootnote}[1]{%
   \setcounter{savefootnote}{\value{footnote}}%
   \setcounter{footnote}{\value{symfootnote}}%
   \ifnum\value{footnote}>8\setcounter{footnote}{0}\fi%
   \let\oldthefootnote=\thefootnote%
   \renewcommand{\thefootnote}{\fnsymbol{footnote}}%
   \footnote{#1}%
   \let\thefootnote=\oldthefootnote%
   \setcounter{symfootnote}{\value{footnote}}%
   \setcounter{footnote}{\value{savefootnote}}%
}

\begin{document}
\allowdisplaybreaks

%
%
%
\begin{center}
{\Large An Efficient Solver for Finite Element-based Constrained Transport in 3D Magnetohydrodynamics Applied to Magnetic Confinement Fusion}
\end{center}
\vspace{-3mm}
\hrulefill
\begin{center}
{Golo A. Wimmer$^{1}$\symfootnote{correspondence to: gwimmer@lanl.gov}, Konstantin Lipnikov$^1$, Ben S. Southworth$^1$, Xian-Zhu Tang$^{1}$}\\
\vspace{4mm}
\today
\end{center}

\begin{abstract}
We present an efficient solver framework for the stiff magnetic wave coupling arising in resistive magnetohydrodynamics (MHD) on realistic tokamak geometries. The approach builds on an implicit-implicit (IMIM) time-splitting that separates fast magnetic waves and anisotropic heat transport from slower acoustic dynamics while retaining full coupling (Krzysik et al. 2026). Within this formulation, the magnetic wave subsystem appears as an anisotropic curl-curl operator, enabling the use of scalable auxiliary-space Maxwell (AMS) multigrid solvers. To exploit this structure at the discrete level, we employ curl-conforming finite element spaces for the magnetic field and design the velocity space to preserve the curl-curl structure induced by the Lorentz-force coupling. The resulting compatible discretization preserves the discrete magnetic divergence constraint while producing linear systems directly amenable to efficient AMS-based solvers. We demonstrate solver efficiency as well as the accuracy and stability of the resulting structure-preserving discretization on fully nonlinear three-dimensional tokamak test cases.
\end{abstract}
\section{Introduction}
In magnetohydrodynamic simulations for magnetic confinement fusion, the fastest time scales are given by heat flux along magnetic field lines and Alfvén waves \cite{jardin2010computational}. Since the dissipative time scales of interest, such as cross-field heat flux or resistive decay, are typically much slower, most MHD codes use implicit schemes to step over the fast scales in tokamak simulations \cite{bonilla2023fully,chacon2008optimal,doecode_12573,hoelzl2021jorek,sovinec2004nonlinear}. However, this comes with the challenge of developing efficient solvers for the resulting complicated and coupled discrete implicit system of equations. In particular, fast parallel heat flux corresponds to an extremely anisotropic scalar diffusion problem, while the Alfvén wave dynamics can be described as an anisotropic curl-curl vector diffusion problem with a nontrivial null-space.

For the latter problem, existing approaches often exploit the tokamak's torus-type domain, which includes in some sense splitting the spatial discretization into a 2D poloidal -- that is $(r,z)$-component in cylindrical coordinates -- part, and a 1D toroidal one -- that is the angular $\phi$ component. The overall 3D problem is then solved for with this structure in mind, considering coupled 2D problems which are tackled using e.g. direct solvers or Krylov methods with simple smoothers \cite{doecode_12573,holod2021enhanced,sovinec2004nonlinear}. However, the corresponding lack in scalability leads to long simulation times for finer mesh resolutions, and paths towards scalable solvers in production codes is an ongoing research topic (e.g., \cite{adams2025fast,quinlan2023towards}). An alternative approach that does not rely on direct solvers is given in \cite{chacon2008optimal}, where assumptions related to a simplified poloidal plane geometry and corresponding structured mesh are used for a space discretization that allows for an efficient, scalable solver for the Alfvén wave dynamics. The solver relies on physics-based preconditioning, in which a Schur complement in the velocity field is formed, which is then solved for using a special stencil for the Alfvén wave operator. Similarly, a physics-based preconditioning approach is considered in the finite element-based work of \cite{bonilla2023fully,ohm2024scalable}, which further allows for general reactor geometries. However, the solvers used therein rely on costly multigrid schemes that result in prohibitively long runtimes in practice. Other approaches include monolithic multigrid methods \cite{abu2023monolithic,adler2021monolithic,shadid2016scalable}, which, however, do not specifically target stiff Alfvén wave couplings, as is the case in tokamak simulations.

Next to the $H^1$- and $C^1$-finite element discretizations considered in the magnetic confinement fusion-specific works of \cite{bonilla2023fully} and \cite{doecode_12573}, respectively, other finite element-based approaches -- which readily allow for general reactor geometries -- include constrained transport, \cite{bochev2001matching,hu2017stable} (and references therin), in which the magnetic field is discretized in a divergence- or curl-conforming space to ensure a form of a discretely satisfied zero-divergence condition. While efficient solvers for this approach have been developed in the context of incompressible MHD in the more recent works \cite{laakmann2022augmented,li2021constrained}, they also do not address the stiff Alfvén-wave coupling relevant to magnetic confinement fusion. This also holds true for the block preconditioning strategy of \cite{ma2016robust}. The work of \cite{zhang2020uniformly} -- which is based on $H^1$-finite elements -- aims to rectify this gap by adding an anisotropic curl-curl vector diffusion term to the magnetic field equation, which, however, is weakly consistent in time but not in space. Altogether, efficient solvers for stiff Alfvén wave coupling in general tokamak geometries remain a challenge and effectively an open area of research.

In this work, we present an efficient solver strategy based on a curl-conforming finite element discretization for the magnetic field. The discretization is akin to our previous work \cite{wimmer2024structure} and weakly preserves the discrete magnetic divergence constraint and energy structure in space, while enabling flexible meshing of realistic tokamak domains. To address stiffness, we employ an implicit-implicit (IMIM) nonlinearly partitioned time discretization that separates solves for the fast magnetic waves, anisotropic heat flux, and slower sound wave dynamics, while retaining full coupling. Combined with a specialized velocity finite element space tailored to the Lorentz-force term, this formulation reduces the magnetic Schur complement to the aforementioned anisotropic curl-curl problem. This structure naturally admits auxiliary-space Maxwell (AMS) preconditioning \cite{hiptmair2007nodal}, which effectively captures the gradient-field null space of the curl-curl operator and enables scalable performance without resorting to direct solvers or prohibitively expensive standard multigrid strategies.

While the intricacies related to the IMIM temporal discretization are explained further in our accompanying manuscript \cite{krzysik2026imim}, here we focus on the space discretization and solver performance. In numerical tests based on second polynomial order spaces, we confirm accuracy and structural properties for the spatial discretization including the modified velocity space. Further, we demonstrate the efficiency of our solver strategy in realistic scenarios with tokamak domains and Grad-Shafranov-type initial equilibria, including a vertical displacement event and a 3D kink instability.

The remainder of this work is structured as follows: In Section \ref{sec_background}, we review the resistive MHD equations and discuss our parameter regime of interest. Further, we review a modified version of an existing space discretization of ours. Based on this, in Section \ref{sec_novelty}, we will first motivate our novel solver strategy, followed by a detailed description thereof. In Section \ref{sec_Numerical_results}, we present and discuss numerical results. Finally, in Section \ref{sec_conclusion}, we review our results and discuss possible future work.
%
\section{Background} \label{sec_background}
In this section, we describe the resistive MHD equations and parameter regime considered in this paper. We further discuss a structure preserving spatial discretization based on our previous work \cite{wimmer2024structure}, as well as the implicit-implicit split time discretization of \cite{krzysik2026imim}.

\subsection{Equations}
Given a 3D plasma domain $\Omega_p$, we consider a non-dimensionalized form of the 3D resistive MHD equations given by \cite{goedbloed2004principles}
\begingroup
\addtolength{\jot}{2mm}
\begin{subequations} \label{3D_MHD_non_discr}
\begin{align}
&\pp{n}{t} + \nabla \cdot (n \mathbf{V}) = 0, \label{cty_eqn_full}\\
&n \left(\pp{\mathbf{V}}{t} + \mathbf{V} \cdot \nabla \mathbf{V} \right) + \beta \nabla ( n T) + \mathbf{B} \times (\nabla \times \mathbf{B}) = \nabla \cdot (Re^{-1} \nabla \mathbf{V}), \label{momentum_eqn_cts} \\
&\frac{n}{\gamma - 1}\left( \pp{T}{t} + \mathbf{V}\cdot \nabla T \right) +  n T \nabla \cdot \mathbf{V} + \nabla \cdot \mathbf{q} = 0, \label{T_eqn_cts}\\
&\pp{\mathbf{B}}{t} + \nabla \times \mathbf{E} = \mathbf{0}, \hspace{1cm} \mathbf{E} = - \mathbf{V} \times \mathbf{B} + S^{-1}(T) \nabla \times \mathbf{B}, \label{B_eqn_cts}
\end{align}
\end{subequations}
\endgroup
with prognostic fields given by the density number $n$, plasma velocity $\mathbf{V}$, temperature $T$, and magnetic field $\mathbf{B}$. In our numerical results section below, we will consider $\Omega_p$ to be either a periodic unit box, or as part of a larger tokamak domain. In the latter case, the overall tokamak domain $\Omega = \Omega_p \cup\Omega_w\cup\Omega_v$ consists of the plasma region $\Omega_p$, together with surrounding wall and vacuum vessel regions $\Omega_w$, $\Omega_v$, respectively -- see Figures \ref{fig_tokamak}, \ref{fig_VDE_kink} in Section \ref{sec_Numerical_results} below. Within $\Omega_{vw} = \Omega_w \cup \Omega_v$, only the magnetic field is evolved according to the resistivity, that is
\begin{equation}
\mathbf{E}|_{\mathbf{x} \in \Omega_{vw}} = S^{-1}(\mathbf{x}) \nabla \times \mathbf{B}.
\end{equation}
For our tokamak simulations, the boundary conditions for the velocity and temperature field along the 2D plasma-wall interface $\Omega_p \cap \Omega_{vw} = \partial\Omega_p$, as well as the boundary condition for the electric field along the outer reactor wall $\partial \Omega$, are given by
\begin{equation}
    \mathbf{V}|_{\partial\Omega_p} = \mathbf{0}, \;\;\;\; T|_{\partial\Omega_p} = T_b,\;\;\;\; \mathbf{E} \times \mathbf{n}|_{\partial \Omega} = \mathbf{0},
\end{equation}
for a specified small constant wall temperature $T_b > 0$.

The adiabatic index is set to $\gamma = \frac{5}{3}$. In our magnetic confinement fusion application of interest, the non-dimensional plasma beta -- that is the reference ratio of internal to magnetic energies -- is given by $\beta \approx 0.03$. Additionally, we define the heat flux $\mathbf{q}$ in terms of an isotropic component and a purely anisotropic component as
\begin{equation}
    \mathbf{q} = - \big(Pe_\Delta^{-1} \mathbf{b} (\mathbf{b} \cdot \nabla T) + Pe_\perp^{-1} \nabla T\big),
\end{equation}
where $\mathbf{b} = \mathbf{B}/|\mathbf{B}|$ \cite{braginskii1965transport}. Finally, the non-dimensional dissipation parameters $Re$, $Pe_\perp$, $Pe_\parallel$ (with $Pe_\Delta = (Pe_\parallel^{-1} - Pe_\perp^{-1})^{-1}$ \footnote{In practice, for large $Pe_\perp$ and small $Pe_\parallel$, we have $Pe_\Delta \approx Pe_\parallel$.}), and $S$ denote the Reynolds, perpendicular, and parallel thermal Péclet and Lundquist numbers, respectively. Note that in this setup, the unit time scale is given by the Alfvén wave time scale, and the dissipative time scales and their corresponding dimensionless coefficients are considered with respect to this scale. For our tokamak-type tests in the numerical results section \ref{sec_Numerical_results} below, we will consider a range of coefficients
\begin{equation}
Re \in \{10^4, 10^5\}, \;\;\; S \in \{10^4, 10^7\} \; (T/T_b)^{-3/2}, \;\;\; Pe_\parallel \in \{20, 100\}, \;\;\; Pe_\perp \in \{2 \times 10^3, 10^5, 10^7\}. \label{non_dim_coeff}
\end{equation}
The Lundquist number is temperature-dependent according to the Spitzer resistivity model \cite{braginskii1965transport}, while for simplicity, we keep the thermal Péclet numbers constant. For the non-plasma regions in the tokamak test cases, we further set $S = 10$ in $\Omega_w$ and $S=10^5$ in $\Omega_v$ throughout. The temperature and magnetic field are non-dimensionalized with respect to their initial values at the magnetic axis, implying that up to small wave perturbations, we have $T \in [T_b, 1]$. In particular, for near-equilibrium initial conditions with small associated flows $\mathbf{V}$, such that we expect an advection CFL number generally smaller than one, the stiffest time scales in \eqref{3D_MHD_non_discr} relate to the Alfvén waves arising from the velocity-magnetic field coupling via the Lorentz force term $\mathbf{B} \times (\nabla \times \mathbf{B})$, as well as the parallel heat flux, and -- to a smaller extent -- the sound wave coupling from the pressure gradient term $\beta \nabla (nT)$. Finally, the remaining dissipative terms generally relate to significantly slower time scales associated with perpendicular heat flux, viscosity, and resistivity.
\subsection{Space discretization} \label{sec_space_discr}
%
We will consider 3D meshes created as extrusions of 2D meshes in a straight direction for the periodic box domain, as well as in the angular $\phi$-direction for the tokamak domain. Triangular and quadrilateral 2D base meshes then give rise to prism and hexahedral meshes, respectively; here, we focus on hexahedral meshes. This choice is motivated by our anisotropic heat flux discretization, for which we employ a mixed finite element method \cite{gunter2007finite} to avoid excessive spurious perpendicular flux. As shown in \cite{wimmer2024fast}, the accuracy of this discretization may deteriorate significantly on triangular and prism meshes.

The discrete magnetic field $\mathbf{B}_h$ is discretized in the curl-conforming Nédélec space $Nc_k^e$ \cite{arnold2014periodic}, for polynomial degree $k$. Following the mixed finite element method of \cite{gunter2007finite}, we set the discrete temperature $T_h$ in the standard continuous Galerkin (CG) space $Q_{k_T}$ equipped with boundary conditions corresponding to $T_b$. Additionally, we consider an auxiliary variable $\zeta_h$ in the discontinuous Galerkin (DG) space $dQ_{k-1}$, which corresponds to the directional gradient $\mathbf{b}\cdot \nabla T_h$ within the heat flux. Similarly, we set the discrete density $n_h$ in the CG space $Q_k$. For the discrete velocity field $\mathbf{V}_h$, in view of mixed finite element stability, we note that the pressure gradient term in \eqref{momentum_eqn_cts} induces a div-grad type coupling with $T_h \in Q_{k_T}$ and $n_h \in Q_k$. Following finite element exterior calculus (FEEC) \cite{arnold2006finite}, the choice of CG fields $T_h$ and $n_h$ then motivates choosing a compatible velocity space such as $\mathring{Nc}_k^e$, that is $Nc_k^e$ equipped with homogenous boundary conditions, for $\mathbf{V}_h$ \cite{boffi2013mixed}. (Linear) mixed finite element stability for the $\mathbf{V}$-$\mathbf{B}$ coupling related to the Lorentz force term is less clear since the cross-product structure leads to degeneracies and does not fit naturally into a FEEC framework. To our knowledge, no standard LBB-type inf-sup stable mixed finite element discretization for the $\mathbf{V}$-$\mathbf{B}$ coupling has been established. Nevertheless, choosing $\mathbf{V}_h \in \mathring{Nc}_k^e$ given $\mathbf{B}_h \in Nc_k^e$ can still be justified from a stability perspective (albeit not via an inf-sup condition); see Appendix \ref{app_2D_MHD} for a brief discussion. In the solver discussion below, we will motivate a modified velocity function space $\mathbb{V}_V$, with corresponding space $\mathring{\mathbb{V}}_V$ including homogeneous boundary conditions, which next to solver considerations will also be based on these stability related observations. 

\begin{remark}[boundary conditions]
\label{remark_BCs}
Since the curl-conforming space $Nc_k^e$ only contains degrees of freedom associated with the vector field's tangential component along the boundary, zero normal components have to be enforced weakly to obtain the no-slip boundary condition $\mathbf{V}|_{\partial \Omega_p} = \mathbf{0}$. An analogous observation holds true for the modified velocity space $\mathbb{V}_V$ to be introduced below. Further, no-slip velocity boundary conditions imply $n_h = n_h|_{t=0}$ at the boundary. Since we only enforce the velocity boundary conditions weakly, $\mathbf{V}_h$ may not be identically zero at the boundary in practice, which in turn may lead to a small buildup of density at the boundary. To avoid this, we weakly enforce a constant density at the plasma boundary, which will lead to a small inconsistency in mass conservation. Note that in more realistic settings, one could consider the Bohm criterion to determine (non-mass conserving) in- and outflow conditions for the velocity field at the boundary (see e.g., \cite{loizu2012boundary}) instead of the no-slip conditions considered here. Correspondingly, we would then instead weakly enforce a density outflux rather than a constant density value.
\end{remark} 

Finally, to avoid possible spurious oscillations related to advection-dominated dynamics as well as from the $\mathbf{V}$-$\mathbf{B}$ coupling, we need to include numerical stabilization terms. For this purpose, we consider continuous interior penalty stabilization terms \cite{burman2005unified,burman2004edge} for the temperature and density fields, and further, following our previous work \cite{wimmer2024structure}, a penalty on $\nabla \times \mathbf{B}$ for the magnetic field equation. For the velocity field, one could also consider a continuous interior penalty term as done in \cite{laakmann2022augmented}; however, for simplicity we penalize jumps related to $\mathbf{V}_h$ itself, akin to a weakly consistent grid-scale viscosity term, rather than jumps in its gradient $\nabla \mathbf{V}_h$. Overall, the resistive MHD equations \eqref{3D_MHD_non_discr} can then be discretized in space with fields $(n_h, \mathbf{V}_h, \boldsymbol{\omega}_h, P_h, T_h, \zeta_h, \mathbf{B}_h) \in (Q_k, \mathring{\mathbb{V}}_V, Nc_k^e, Q_k, Q_{k_T}, dQ_{k-1}, Nc_k^e)$ such that
\begingroup
\addtolength{\jot}{3mm}
\begin{subequations} \label{3D_MHD_discr}
\begin{align}
&\left\langle \chi, \pp{n_h}{t}\right\rangle - \left\langle \nabla \chi, \mathbf{V}_h n_h \right\rangle + L_{\kappa_n}(\chi, n_h) = 0& \forall \chi \in Q_k(\Omega_p)\\
&\left\langle n_h \mathbf{w}, \pp{\mathbf{V}_h}{t} + \boldsymbol{\omega}_h \times \mathbf{V}_h + \tfrac{1}{2}\nabla P_h + \beta \nabla ( n_h T_h) + \mathbf{B}_h \times (\nabla \times \mathbf{B}_h) \right\rangle \nonumber \\
& \hspace{1cm} + L(\mathbf{w}, \mathbf{V}_h) + L_{\kappa_V}\big(f(\mathbf{w}), f(\mathbf{V}_h)\big)=0 & \forall \mathbf{w} \in \mathring{\mathbb{V}}_V(\Omega_p), \label{V_eqn_discr} \\
& \left\langle \nu, P_h - \tfrac{1}{2}|\mathbf{V}_h|^2 \right\rangle = 0 &\forall \nu \in Q_k(\Omega_p),\\
& \left\langle \boldsymbol{\epsilon}, \boldsymbol{\omega}_h \right\rangle - \left\langle \nabla \times \boldsymbol{\epsilon}, \mathbf{V}_h \right\rangle = 0 &\forall \boldsymbol{\epsilon} \in Nc_k^e(\Omega_p)\\
&\left\langle \frac{n_h \eta}{\gamma - 1}, \pp{T_h}{t} \right\rangle - \left\langle \nabla(\eta  n_h T_h), \mathbf{V}_h \right\rangle \nonumber \\
&\hspace{1cm}+ \left\langle \mathbf{b}_h \cdot \nabla \eta, Pe_\Delta^{-1} \zeta_h \right\rangle + \left\langle \nabla \eta, Pe_\perp^{-1} \nabla T_h \right\rangle + L_{\kappa_T}(\eta, n_h, T_h) = 0 & \forall \eta \in \mathring{Q}_k(\Omega_p), \label{T_eqn_discr}\\
&\langle \phi, \zeta_h - \mathbf{b}_h\cdot \nabla T_h \rangle = 0 &\forall \phi \in dQ_{k-1}(\Omega_p), \\
&\left\langle \mathbf{\Sigma}, \pp{\mathbf{B}_h}{t} \right\rangle_{\!\!\Omega} - \left\langle \nabla \times \mathbf{\Sigma}, \mathbf{V}_h \times \mathbf{B}_h\right\rangle \nonumber \\
&\hspace{1cm}+ \left\langle \nabla \times \mathbf{\Sigma}, S^{-1}(T_h) \nabla \times \mathbf{B}_h\right\rangle_\Omega + L_{\kappa_B}(\mathbf{\Sigma}, \mathbf{B}) = 0&\forall \mathbf{\Sigma} \in Nc_k^e(\Omega), \label{B_eqn_discr}
\end{align}
\end{subequations}
\endgroup
where $\mathring{Q}_k$ in the temperature equation is the discrete CG space equipped with homogeneous boundary conditions. Here $\langle ., .\rangle$ and $\langle ., .\rangle_\Omega$ denote $L^2$ inner products over $\Omega_p$ and $\Omega$, respectively (with $\Omega = \Omega_p$ for our non-tokamak test case). Further, we have introduced two auxiliary variables $\boldsymbol{\omega}_h$, $P_h$ in the discretization of the velocity advection term $\mathbf{V} \cdot \nabla \mathbf{V} = (\nabla \times \mathbf{V}) \times \mathbf{V} + \frac{1}{2}\nabla |\mathbf{V}|^2$, with $\boldsymbol{\omega}_h$ and $P_h$ corresponding to the vorticity $\nabla \times \mathbf{V}$ and $\tfrac{1}{2}|\mathbf{V}|^2$, respectively. Additionally, $L(\mathbf{w}, \mathbf{V}_h)$ denotes a standard non-conforming, symmetric interior penalty based discretization of the viscosity \cite{di2011mathematical}, given by
\begingroup
\addtolength{\jot}{2mm}
\begin{align}
L(\mathbf{w}, \mathbf{V}_h) = &- \langle \nabla \mathbf{V}_h, \nabla \mathbf{w} \rangle + \int_\Gamma \left( \left\{ \nabla \mathbf{V}_h \right\} : [\![ \mathbf{w} ]\!] + \left\{ \nabla \mathbf{w} \right\} : [\![ \mathbf{V}_h ]\!] \right)dS - \int_\Gamma \frac{\kappa}{h_e} [\![ \mathbf{V}_h ]\!] \cdot [\![ \mathbf{w} ]\!] dS \nonumber\\
&\hspace{-5mm} + \int_{\partial \Omega_p} \left(\left( \nabla \mathbf{V}_h \cdot \mathbf{n} \right) \cdot \mathbf{w} + \left( \nabla \mathbf{w} \cdot \mathbf{n} \right) \cdot (\mathbf{V}_h - \mathbf{V}_{bc})\right) ds - \int_{\partial \Omega_p} \frac{\kappa}{h_e} (\mathbf{V}_h - \mathbf{V}_{bc}) \cdot \mathbf{w} \; ds,
\end{align}
\endgroup
where $\Gamma$ denotes the set of all interior facets within $\Omega_p$, $\{\nabla \mathbf{w}\} = (\nabla \mathbf{w}^+ + \nabla \mathbf{w}^-)/2$, $[\![\mathbf{w}]\!] = \mathbf{w}^+ \otimes \mathbf{n}^+ + \mathbf{w}^- \otimes \mathbf{n}^-$, $\kappa > 0$ is a stabilization parameter, and $h_e$ is a mesh facet length scale defined by $\{|K|\}/|F|$, for facets $F$ and their neighboring cells $K$. Further, $\mathbf{V}_{bc}$ denotes a specified boundary value along $\partial\Omega_p$, which in this work is set to zero for test cases including boundaries.

Finally, the penalty terms are given by
\begingroup
\begin{subequations} \label{CIP_terms}
\begin{align}
&L_{\kappa_n}(\chi, n_h) = \int_\Gamma h_e^2 \kappa_n [\![\nabla n_h]\!] \cdot [\![\nabla \chi]\!]dS + \int_{\partial \Omega_p} k_n^{bc} \chi(n - n_{bc}) dx \label{CIP_n}\\ 
&L_{\kappa_V}\big(f(\mathbf{w}), f(\mathbf{V})\big) = \int_\Gamma \kappa_V [\![f(\mathbf{w})]\!]\cdot[\![f(\mathbf{V})]\!]dS, \label{IP_V}\\
&L_{\kappa_T}(\eta, n_h, T_h) = \frac{1}{\gamma - 1}\int_\Gamma n_h h_e^2 \kappa_T [\![\nabla T_h]\!] \cdot [\![\nabla \eta]\!]dS, \\
& L_{\kappa_B}(\mathbf{\Sigma}, \mathbf{B}) = \int_\Gamma h_e^2 \kappa_B [\![\nabla \times \mathbf{\Sigma}]\!] \cdot[\![\nabla \times \mathbf{B}_h]\!] dS,
\end{align}
\end{subequations}
\endgroup
where $\kappa \in \{\kappa_n, \kappa_V, \kappa_T, \kappa_B\}$ are interior penalty stabilization parameters (see \cite{burman2004edge}). Further, as discussed in Remark \ref{remark_BCs}, the density equation stabilization term $L_{\kappa_n}$ also contains a term (with $\kappa_p^{bc} > 0$) weakly enforcing $n_{bc} = n|_{t=0}$ at the plasma boundary\footnote{In the numerical results section below, for implementational reasons related to weakly enforced boundary conditions, the density field is also evolved according to $\langle \chi, \partial_t n_h \rangle = 0$ within the wall-vacuum vessel region $\Omega_{vw}$. For this reason, an additional damping term akin to the last term in \eqref{CIP_n} is included throughout $\Omega_{wv}$.}. Additionally, $f$ in \eqref{IP_V} is a function related to the modified velocity space and will be specified in the solver section below.

Note that equations \eqref{3D_MHD_discr} contain slight deviations from our previous work \cite{wimmer2024structure}, including additional stabilization terms and the mixed heat flux formulation of \cite{gunter2007finite}. These changes relate to improved stability in simulations for tokamak domains. Further, we consider a different velocity space and discretization of the velocity advection terms here; these changes in turn relate to our solver strategy presented below, which relies on a modified velocity function space. In particular, we found the upwind-discretized form of the curl-term $(\nabla \times \mathbf{V}) \times \mathbf{V}$ (based on \cite{natale2016compatible}) to not work well with the modified space. Instead, we consider an auxiliary variable-based setup, which e.g., as shown in the context of the shallow water equations \cite{mcrae2014energy,wimmer2020energy}, ensures a consistently evolving discrete total energy. Similarly to \cite{wimmer2024structure}, statements on structure preservation can still be made for the discrete equations \eqref{3D_MHD_discr}, including the preservation of a weakly defined magnetic field divergence, as well as energetic consistency up to the interior penalty term in the density equation. Further details are provided in Appendix \ref{app_structure}.
%
\section{Solver strategy} \label{sec_novelty}
Having described our space discretization, we next introduce our solver strategy, which is targeted at the stiff magnetic-velocity field coupling resulting from our parameter regime of interest, with internal to magnetic energy ratio $\beta \ll 1$ and slow dissipative time scales of interest \eqref{non_dim_coeff}, up to the parallel heat flux. We will first motivate our strategy in Section \ref{sec_motivation}, discussing the Alfvén wave coupling as an anisotropic curl-curl vector diffusion operator. This is followed by a detailed description of our solver strategy in Section \ref{sec_solver_procedure}, which consists of three main parts. Finally, we will present the fully discrete scheme and describe our block preconditioning procedure in Section \ref{sec_full_scheme}.
\subsection{Motivation} \label{sec_motivation}
Before discretization, and considered in isolation to all remaining terms, a linearized form of the $\mathbf{V}$-$\mathbf{B}$-coupling (within the plasma domain $\Omega_p$) can be written as
\begingroup
\begin{subequations} \label{BV_system}
\begin{align}
    \pp{\mathbf{V}}{t} + \mathbf{B}_0 \times (\nabla \times \mathbf{B}) = \mathbf{0}, \\
    \pp{\mathbf{B}}{t} + \nabla \times (\mathbf{B}_0 \times \mathbf{V}) = \mathbf{0},
\end{align}
\end{subequations}
for some known magnetic field $\mathbf{B}_0$ defined in the same space as $\mathbf{B}$. After discretizing in space using suitably regular spaces $\mathbb{U}_B$ and $\mathbb{U}_V$, and upon eliminating either the magnetic or the velocity field, the mixed system gives rise to two possible equations of the form
\begingroup
\addtolength{\jot}{2mm}
\begin{subequations}
\begin{align}
    &\left\langle \boldsymbol{\sigma}, \pp{^2\mathbf{B}}{t^2} \right\rangle + \left\langle  \mathbf{B}_0 \times (\nabla \times \boldsymbol{\sigma})), \mathbf{B}_0 \times (\nabla \times \mathbf{B})\right\rangle = 0 &\forall \boldsymbol{\sigma} \in \mathbb{U}_B, \label{B_wave_eqn} \\
    &\left\langle \mathbf{v}, \pp{^2\mathbf{V}}{t^2} \right\rangle + \left\langle  \nabla \times (\mathbf{B}_0 \times \mathbf{v}), \nabla \times (\mathbf{B}_0 \times \mathbf{V})\right\rangle = 0 &\forall \mathbf{v} \in \mathbb{U}_V.\label{V_wave_eqn}
\end{align}
\end{subequations}
\endgroup
The magnetic field equation \eqref{B_wave_eqn} contains a symmetric positive semi-definite operator involving cross products applied to curls, while the opposite order holds true for \eqref{V_wave_eqn}. Consequently, the operator in \eqref{B_wave_eqn} contains a null-space of gradient fields, while for \eqref{V_wave_eqn} it is a null-space of cross-products that can be described as gradient fields. It is unclear how to construct solvers that effectively treat the latter type of null-space, and in particular custom-made stencils such as employed in the finite-volume based work of \cite{chacon2008optimal} do not easily carry over to this finite element-based context. With this in mind, we will consider an approach based on eliminating the velocity in favor of a Schur complement corresponding to \eqref{B_wave_eqn}.

We next consider our space discretization \eqref{3D_MHD_discr} for the full MHD equations. Applying a standard implicit scheme such as the midpoint rule in time leads to a fully coupled, discrete system of nonlinear equations, which can be solved for using iterative procedures such as the (quasi-)Newton method. In each iteration, we are then required to solve a linear system of equations. For illustration purposes, if we assume an advection CFL less than 1, skip the auxiliary variables $\boldsymbol{\omega}_h$, $P_h$, and assume $\zeta_h$ has been eliminated, an approximate version of the resulting system takes the form
\begin{equation}
\begin{pmatrix}
  \frac{1}{\alpha\delta t}M^V_{n_0} \!+\! L_{\kappa_V} \!+\! L_{Re} & \beta G_{T_0} & \beta G_{n_0} & C_{B_0}\\
  - G_{n_0}^\top & \frac{1}{\alpha\delta t}M^n \!+\! L_{\kappa_n}  & 0 & 0 \\
  -G_{n_0T_0}^\top & 0 & \!\!\!\!\tfrac{1}{\alpha\delta t(\gamma-1)}M^T_{n_0} \!+\! L_{\kappa_T} \!+\! Q & 0 \\
  -C_{B_0}^\top & 0 & 0 & \!\!\!\!\frac{1}{\alpha\delta t}M^B \!+\! L_{\kappa_B} \!+\! L_S
\end{pmatrix}
\begin{pmatrix}
  \delta \mathbf{V}\\
  \delta n \\
  \delta T \\
  \delta \mathbf{B}
\end{pmatrix}
= R(\mathbf{z}), \label{MHD_system_full}
\end{equation}
for mass matrix blocks $M^x$ with field indicator $x \in (V, n, T, B)$, where a subscript $n_0$ corresponds to mass matrices weighted by a known density field $n_0$ arising in the linearization. Further, $G$ corresponds to a discrete gradient operator, where again subscripts by $n_0$ or a known temperature field $T_0$ denote weights. $C_{B_0}$ corresponds to the discrete directional curl operator $\mathbf{B}_0 \times (\nabla \times (\cdot))$, for a known magnetic field $\mathbf{B}_0$. Additionally, the diagonal entries $L_\kappa$ correspond to penalty stabilization terms, and $L_{Re}$, $Q$, $L_S$ to the viscosity, heat flux (linearized with respect to $\mathbf{B}_0$ and with $\zeta_h$ incorporated), and resistivity. $R$ denotes a residual with known state $\mathbf{z}$ and nonlinear update $\delta \mathbf{z} = (\delta n, \delta \mathbf{V}, \delta T, \delta \mathbf{B})^\top$ (where $^\top$ denotes the transpose). Finally, $\delta t$ denotes the time step, and $\alpha$ a factor associated with the time discretization.


The approximate Jacobian \eqref{MHD_system_full} -- which includes the main stiff terms -- contains a block diagonal structure in $(\delta n, \delta T, \delta \mathbf{B})$. This suggests a block preconditioning strategy with a Schur complement in $\delta \mathbf{V}$, as considered in \cite{chacon2008optimal}, but this is in contrast to our favored approach of eliminating $\mathbf{V}$ to a Schur complement in $\mathbf{B}$ \eqref{B_wave_eqn}. However, Schur complement approximation in $\delta \mathbf{B}$ with respect to the larger system in \eqref{MHD_system_full} is less straightforward due to the sound wave coupling blocks $\beta G_{T_0}$, etc. Even if we ignore the sound wave coupling, as well as velocity dissipation terms $L_{\kappa_V}$, $L_{Re}$, the discrete anisotropic curl-curl operator will be of the block-matrix form $C_{B_0}^\top (M_{n_0}^V)^{-1}C_{B_0}$, which may not be well-approximated by the weakly defined version in \eqref{B_wave_eqn}. We will address each of these concerns in the following section.

\subsection{Approach based on Schur complement in magnetic field} \label{sec_solver_procedure}
%
Following the discussion above, our solver strategy involves three main parts, which we will discuss here:
\begin{itemize}
    \vspace{-2mm}
    \item a suitable split-time discretization to avoid thermodynamical coupling terms in the Schur complement in $\mathbf{B}$,
    \vspace{-2mm}
    \item a suitable velocity space to allow for an adequate appproximation to the Schur complement,
    \vspace{-2mm}
    \item an efficient solver for the approximate Schur complement.
    \vspace{-2mm}
\end{itemize}
For this purpose, we start with the linearized $\mathbf{V}$-$\mathbf{B}$ coupling equations \eqref{BV_system} and work our way up to the full resistive MHD equations, discussing the three concepts in reverse order in this section. In Section \ref{sec_full_scheme}, we then put everything together, including remarks on additional details not covered in this section.

\subsubsection{Preconditioning with Schur complement in $\mathbf{B}$} Using a suitable implicit time discretization, a discrete form of \eqref{BV_system}, whose space discretization corresponds to our discrete resistive MHD setup \eqref{3D_MHD_discr}, can be written as
\begingroup
\begin{subequations} \label{V_B_coupling_discr}
\begin{align}
    &\left\langle \mathbf{w}, \mathbf{V}_h^{n+1} + \alpha\delta t \mathbf{B}_{h,0} \times (\nabla \times \mathbf{B}_h^{n+1})\right\rangle = \left\langle \mathbf{w}, \mathbf{r}_V\right\rangle &\forall \mathbf{w} \in \mathbb{V}_V(\Omega_p), \\
    &\left\langle \mathbf{\Sigma}, \mathbf{B}_h^{n+1} \right\rangle + \alpha\delta t \left\langle \nabla \times \mathbf{\Sigma}, \mathbf{B}_{h,0} \times \mathbf{V}_h^{n+1} \right\rangle = \left\langle\mathbf{\Sigma}, \mathbf{r}_B\right\rangle & \forall \mathbf{\Sigma} \in Nc_k^e(\Omega_p),
\end{align}
\end{subequations}
\endgroup
for a discrete velocity space $\mathbb{V}_V$ to be determined (and ignoring for now the associated boundary conditions), unknown next time level fields $\mathbf{V}_h^{n+1}$, $\mathbf{B}_h^{n+1}$ to be solved for, and where $\mathbf{r}_V \in \mathbb{V}_V(\Omega_p)$, $\mathbf{r}_B \in Nc_k^e(\Omega_p)$ are known right-hand side functions that can be computed based on the known fields $\mathbf{V}_h^n$, $\mathbf{B}_h^n$ of the current time level. Further, $\delta t$ is the time step size, and $\alpha$ is a factor associated with the choice of implicit time discretization. Taking a Schur complement in the magnetic field then corresponds to eliminating the velocity field, which in weak variational form can be written as
\begin{align}
    &\left\langle \mathbf{\Sigma}, \mathbf{B}_h^{n+1} \right\rangle + \alpha^2\delta t^2 \left\langle \mathbf{B}_{h,0} \times (\nabla \times \mathbf{\Sigma}), P_{\mathbb{V}_V}\big(\mathbf{B}_{h,0} \times (\nabla \times \mathbf{B}^{n+1}_h)\big) \right\rangle = R_{S_B}(\mathbf{\Sigma}) & \forall \mathbf{\Sigma} \in Nc_k^e(\Omega_p). \label{Schur_B_project}
\end{align}
Here $P_{\mathbb{V}_V}$ denotes the $L^2$-projection into $\mathbb{V}_V$, and $R_{S_B}$ corresponds to a known right-hand side
\begin{equation}
    R_{S_B}(\mathbf{\Sigma}) = \left\langle\mathbf{\Sigma}, \mathbf{r}_B\right\rangle -\alpha \delta t \left\langle \nabla \times \mathbf{\Sigma}, \mathbf{B}_{h,0} \times \mathbf{r}_V \right\rangle.
\end{equation}
Next, we assume $\mathbb{V}_V$ to be chosen such that the projection of $\mathbf{B}_{h,0} \times (\nabla \times \mathbf{B}_h^{n+1})$ therein approximately corresponds to an embedding for any possible solution $\mathbf{B}_h^{n+1}$, that is,
\begin{align}
    P_{\mathbb{V}_V}(\mathbf{B}_{h,0} \times (\nabla \times \mathbf{B}_h^{n+1})) \approx \mathbf{B}_{h,0} \times (\nabla \times \mathbf{B}_h^{n+1}) &&\forall \mathbf{B}_h^{n+1} \in Nc_k^e(\Omega_p). \label{P_V_approx}
\end{align}
Note that this is an important assumption corresponding to the second of the three solver strategy concepts discussed in this section, and this will be described in more detail further below.

In this case, the left-hand side of \eqref{Schur_B_project}, which can be seen as the bilinear form $s_b(\mathbf{\Sigma}, \mathbf{B}_h^{n+1})$ corresponding to the Schur complement $S_B$ in the magnetic field, can be approximated by $s_B'$ of the form
\begingroup
\addtolength{\jot}{2mm}
\begin{align}
    s_B \approx s_B' = &\left\langle \mathbf{\Sigma}, \mathbf{B}_h^{n+1} \right\rangle + \alpha^2\delta t^2 \left\langle \mathbf{B}_{h,0} \times (\nabla \times \mathbf{\Sigma}), \mathbf{B}_{h,0} \times (\nabla \times \mathbf{B}_h^{n+1}) \right\rangle \nonumber \\
    = &\left\langle \mathbf{\Sigma}, \mathbf{B}_h^{n+1} \right\rangle + \alpha^2\delta t^2 \left\langle \nabla \times \mathbf{\Sigma}, \boldsymbol{\Omega}_{h,0}\nabla \times \mathbf{B}_h^{n+1} \right\rangle & \forall \mathbf{\Sigma} \in Nc_k^e(\Omega_p), \label{Schur_B}
\end{align}
\endgroup
for tensor
\begin{equation}
    \boldsymbol{\Omega}_{h,0} = |\mathbf{B}_{h,0}|^2 I - \mathbf{B}_{h,0}\mathbf{B}_{h,0}^T = |\mathbf{B}_{h,0}|^2(I - \mathbf{b}_{h,0} \mathbf{b}_{h,0}^T), \label{Omega_0}
\end{equation}
where $\mathbf{b}_{h,0}$ is the unit vector in the direction of $\mathbf{B}_{h,0}$. In other words, in line with the transverse nature of Alfvén waves \cite{goedbloed2004principles}, $\boldsymbol{\Omega}_{h,0}$ corresponds to a projection into the perpendicular plane relative to the magnetic field $\mathbf{B}_{h,0}$, scaled by $\mathbf{B}_{h,0}^2$.

Note that for $\boldsymbol{\Omega}_{h,0} \equiv I$, \eqref{Schur_B} corresponds to the standard curl-curl problem in electromagnetism, for curl-conforming field $\mathbf{B}_h^{n+1}$ to be solved for. In particular, Auxiliary-space Maxwell Solvers (AMS) \cite{hiptmair2007nodal}, which exploit ideas from FEEC to explicitly construct the near-nullspace consisting of gradient fields in the CG space $Q_k$, are well-suited for this problem. Although AMS is designed for isotropic curl-curl problems, we show in the numerical results section that AMS remains effective for the anisotropic curl-curl problem \eqref{Schur_B} arising in tokamak-type magnetic field configurations.

\begin{remark}[local projection stabilization] $s_B$ and its approximation $s_B'$ can be shown to relate to each other via
\begin{align}
    s_B' = s_B - \alpha^2 \delta t^2 \left\langle(\iota - P_{\mathbb{V}_V})(\mathbf{B}_{h,0} \times (\nabla \times \mathbf{\Sigma})),(\iota - P_{\mathbb{V}_V})(\mathbf{B}_{h,0} \times (\nabla \times \mathbf{B}_h^{n+1}))\right\rangle,
\end{align}
for any $\mathbf{\Sigma} \in Nc_k^e(\Omega_p)$, and where $\iota$ denotes the identity. In other words, $s_B$ consists of the form $s_B'$ corresponding to a positive-definite operator, together with an additional term that captures local scales that are not resolved by $\mathbb{V}_V$. The term is akin to the local projection stabilization method \cite{braack2006local}, except that here, it acts as antidiffusion rather than diffusion. We note that the additional diffusive term added to the incompressible MHD discretization in \cite{zhang2020uniformly} can be interpreted as a counter to this antidiffusion term, while here, we instead approximately remove it through our choice of velocity space.
\end{remark}

\subsubsection{Modified velocity function space} In our discussion on AMS-based preconditioning above, we crucially assumed the approximation \eqref{P_V_approx}, which in turn relies on the choice of velocity function space $\mathbb{V}_V$. In Section \ref{sec_space_discr}, we argued for $\mathbb{V}_V = Nc_k^e$ in view of mixed finite element stability with respect to the pressure gradient term and, to a lesser extent, the Lorentz force term. In practice, we find that for $\mathbb{V}_V = Nc_k^e$, the approximation in \eqref{P_V_approx} is not good, and consequently \eqref{Schur_B} together with AMS is not an effective preconditioner to the Schur complement equation \eqref{Schur_B_project}. Similarly, this holds true when setting $\mathbb{V}_V$ to the div-conforming Raviart-Thomas-Nédélec space $Nc_k^f$. Note that following FEEC, for $\mathbf{B}_h^{n+1} \in Nc_k^e$, its curl $\nabla \times \mathbf{B}_h^{n+1}$ is a field in the div-conforming space $Nc_k^f$ and therefore the projected term $\mathbf{B}_{h, 0} \times (\nabla \times \mathbf{B}_h^{n+1})$ consists of a cross product of fields in $Nc_k^e$ and $Nc_k^f$.\footnote{In principle, one could also consider discretizations with $\mathbf{B}_{h,0}$ in a different space than $\mathbf{B}_h$ itself, as considered e.g., in \cite{hu2021helicity} in the context of helicity preservation; we could not find an efficient solver strategy ensuring \eqref{P_V_approx} with such an approach.} The cross product can therefore plausibly be seen to be embedded in a higher order vector DG space, and in practice it can be shown (numerically) that \eqref{P_V_approx} holds true -- i.e. $s_B'$ in \eqref{Schur_B} is a good approximation to the Schur complement variational form $s_B$ -- e.g., for $\mathbb{V}_V = [dQ_{2k}]^3$ (or even $[dQ_{k+1}]^3$ for the case $k=2$). However, this leads to velocity spaces with a relatively large number of degrees of freedom, and consequently to an imbalanced ratio of degrees of freedom of $\mathbb{V}_V$ versus the magnetic field space $Nc_k^e$. In practice, this in turn leads to mixed finite element stability problems (cf. \cite{boffi2013mixed,cotter2012mixed}) with corresponding serious spurious numerical modes.

To avoid such stability issues while still satisfying \eqref{P_V_approx}, we further build on the observation that $\mathbf{B}_{h, 0} \times (\nabla \times \mathbf{B}_h^{n+1})$ is a cross product of a known curl-conforming and an unknown div-conforming field. In particular, by design, the space
\begin{equation}
\mathbb{V}_{V,0}'(\Omega_p) \coloneqq \left\{\mathbf{B}_{h,0} \times \boldsymbol{w} : \boldsymbol{w}\in Nc_k^f(\Omega_p)\right\},
\end{equation}
is such that
\begin{equation}
    \mathbf{B}_{h, 0} \times (\nabla \times \mathbf{B}_h^{n+1}) \in \mathbb{V}'_V(\Omega_p). \label{V_space_mod_degenerate}
\end{equation}
However, while $\mathbb{V}_V'$ satisfies \eqref{P_V_approx} exactly and would therefore be ideal for our proposed AMS-based preconditioning approach, it may be a degenerate space in the sense that there may exist $\boldsymbol{w} \in Nc_k^f$ such that $\mathbf{B}_{h,0} \times \boldsymbol{w} = \mathbf{0}$. A simple way to avoid this is given by a shift, leading to the following definition.
\begin{definition}[Modified velocity space]
    Given a discrete, curl-conforming magnetic field $\mathbf{B}_h \in Nc_k^e(\Omega)$, the discrete modified velocity space $\mathbb{V}_V[\mathbf{B}_h](\Omega_p)$ is defined by
    \begin{equation}
    \mathbb{V}_{V}[\mathbf{B}_h](\Omega_p) \coloneqq \left\{\mathbf{B}_h \times \boldsymbol{v} + c_0 \boldsymbol{v}: \boldsymbol{v}\in Nc_k^f(\Omega_p)\right\}, \label{V_space_mod_0}
\end{equation}
for a shift factor $c_0 \neq 0$. Further, we write functions $\mathbf{w}$ therein as
\begin{equation}
    \mathbf{w} = \mathbf{w}(\mathbf{B}_h) \coloneqq \mathbf{B}_h\times\boldsymbol{v} + c_0\boldsymbol{v}\in \mathbb{V}_V[\mathbf{B}_h](\Omega_p) \label{V_space_mod_0_func}
\end{equation}
\end{definition}
%
In practice, we find $c_0 = 1$ to work well, and we will use this value throughout the numerical results section. Note that here we assume non-dimensionalized quantities; if a dimensional magnetic field was used instead, the cross-product would need to be scaled appropriately to ensure correct units for the overall velocity field. Additionally, note that functions in the finite element space $\mathbb{V}_{V}[\mathbf{B}_h](\Omega_p)$ are expressions of the form $\mathbf{B}_h\times\boldsymbol{v} + c_0 \boldsymbol{v}$, rather than a projection thereof into some standard finite element space. Therefore, in the following, any appearances of functions in $\mathbb{V}_{V}[\mathbf{B}_h](\Omega_p)$ can directly be expressed in their expanded form \eqref{V_space_mod_0_func} with a corresponding field $\boldsymbol{v} \in Nc_k^f(\Omega_p)$. Indeed, in practice, when assembling and solving for the discrete MHD equations, we will do so using test and trial functions in $Nc_k^f$. The full velocity field $\mathbf{V}_h$ is then only computed for visualization purposes as a diagnostic via a projection of $\mathbf{B}_h \times \boldsymbol{\mathcal{U}}_h + c_0 \boldsymbol{\mathcal{U}}_h$ into a suitable output space (such as $[dQ_k]^3$), where $\boldsymbol{\mathcal{U}}_h \in Nc_k^f(\Omega_p)$ is the prognostic field that we solve for and keep track of through the simulation.

\begin{remark}[Well-definedness and further discussion on $\mathbb{V}_V$]
As stated before, unlike the space \eqref{V_space_mod_degenerate}, the space \eqref{V_space_mod_0} is no longer degenerate. By using the $3 \times 3$ matrix form of the cross product, one can compute the inverse of $\mathbf{B}_{h,0} \times \boldsymbol{v} + c_0 \boldsymbol{v}$ as
\begin{equation}
    \frac{1}{c_0^2 + |\mathbf{B}_{0,h}|^2}\left(c_0 \boldsymbol{v} - \mathbf{B}_{h,0} \times \boldsymbol{v} + c_0^{-1} \mathbf{B}_{h,0}(\mathbf{B}_{h,0} \cdot \boldsymbol{v})\right). \label{V_mod_inv}
\end{equation}
Further, by definition, $\mathbb{V}_V[\mathbf{B}_h](\Omega_p)$ is of equal size as the div-conforming space $Nc_k^f(\Omega_p)$, and therefore for hexahedral meshes, it is approximately of equal size as the magnetic field's curl-conforming space $Nc_k^e(\Omega_p)$. Therefore, possible instabilities related to a strongly imbalanced ratio of degrees of freedom is less of a concern. Finally, we note that from the definition \eqref{V_space_mod_0} and from \eqref{V_mod_inv}, zero boundary conditions for $\mathbf{V}_h = \boldsymbol{\mathcal{U}}_h + c_0 \boldsymbol{\mathcal{U}}_h \in \mathbb{V}_{V}[\mathbf{B}_h](\Omega_p)$ are equivalent to zero boundary conditions for $\boldsymbol{\mathcal{U}}_h$. Further, for Bohm-type boundary conditions (see Remark \ref{remark_BCs}) of the form $\mathbf{V}_h|_{\partial \Omega_p} \propto \mathbf{B}_h|_{\partial \Omega_p}$, we note that the cross products in \eqref{V_space_mod_0} and \eqref{V_mod_inv} disappear when substituting fields $\mathbf{v}_h \propto \mathbf{B}_h$, and in particular we see that in this case, the boundary conditions for $\mathbf{V}_h$ and $\boldsymbol{\mathcal{U}}_h$ are also equivalent (up to the constant coefficient $c_0$).

\end{remark}

Having defined our modified, non-degenerate velocity space \eqref{V_space_mod_0}, it remains to investigate the implications of the regularizing shift $c_0 \boldsymbol{v}$ on our preconditioning strategy. For the remainder of the velocity space-related discussion based on the linearised $\mathbf{V}$-$\mathbf{B}$ coupling equations \eqref{V_B_coupling_discr}, we consider
\begin{equation}
\mathbb{V}_{V,0}(\Omega) \coloneqq \mathbb{V}_V[\mathbf{B}_{h,0}](\Omega_p).
\end{equation} 
Equations \eqref{V_B_coupling_discr} now read
\begingroup
\addtolength{\jot}{2mm}
\begin{subequations} \label{V_B_eqns_V_mod}
\begin{align}
    &\left\langle \mathbf{w}(\mathbf{B}_{h,0}), \mathbf{V}_h^{n+1}(\mathbf{B}_{h,0}) + \alpha\delta t \mathbf{B}_{h,0} \times (\nabla \times \mathbf{B}_h^{n+1})\right\rangle = \langle \mathbf{w}, \mathbf{r}_V\rangle &\forall \mathbf{w} \in \mathbb{V}_{V,0}(\Omega_p), \label{V_eqn_V_mod}\\
    &\left\langle \mathbf{\Sigma}, \mathbf{B}_h^{n+1} \right\rangle + \alpha\delta t \left\langle \nabla \times \mathbf{\Sigma}, \mathbf{B}_{h,0} \times \mathbf{V}_h^{n+1}(\mathbf{B}_{h,0}) \right\rangle = \langle\mathbf{\Sigma}, \mathbf{r}_B\rangle & \forall \mathbf{\Sigma} \in Nc_k^e(\Omega_p), \label{B_eqn_V_mod}
\end{align}
\end{subequations}
\endgroup
which is still in Galerkin form (rather than Petrov-Galerkin), in that both test and trial functions belong to the same modified velocity space $\mathbb{V}_{V,0}(\Omega_p)$.
\begin{proposition} \label{prop_Schur}
Consider equations \eqref{V_B_eqns_V_mod}, together with a velocity field $\boldsymbol{\mathcal{U}}_h^{n+1} \in Nc_k^f(\Omega_p)$ defined such that
    \begin{equation}
    \mathbf{V}_h^{n+1}(\mathbf{B}_{h,0}) = \mathbf{B}_{h,0} \times \boldsymbol{\mathcal{U}}_h^{n+1} + c_0 \boldsymbol{\mathcal{U}}_h^{n+1}.\label{V_mod_subst}
    \end{equation}
Then equations \eqref{V_B_eqns_V_mod} admit an equation in the magnetic field $\mathbf{B}_h^{n+1}$ whose variational form is given by
\begingroup
\addtolength{\jot}{2mm}
\begin{align}
    &\left\langle \mathbf{\Sigma}, \mathbf{B}_h^{n+1} \right\rangle + \alpha^2\delta t^2 \left\langle \mathbf{B}_{h,0} \times (\nabla \times \mathbf{\Sigma}), \mathbf{B}_{h,0} \times (\nabla \times \mathbf{B}_h^{n+1})\right\rangle \nonumber \\
    &\hspace{5mm}= R_B(\mathbf{\Sigma}) -\alpha \delta t \langle \nabla \times \mathbf{\Sigma}, \mathbf{B}_{h,0} \times \mathbf{r}_V\rangle  + \alpha \delta t c_0 \langle \nabla \times \mathbf{\Sigma}, \boldsymbol{\nu}_{res}\rangle& \forall \mathbf{\Sigma} \in Nc_k^e(\Omega_p), \label{B_Schur_eqn_V_mod}
\end{align}
\endgroup
where $\boldsymbol{\nu}_{res} \in Nc_k^f(\Omega_p)$ is a weak velocity equation residual defined by
\begin{align}
    \langle \mathbf{v}, \mathbf{V}_h^{n+1}(\mathbf{B}_{h,0}) + \alpha\delta t \mathbf{B}_{h,0} \times (\nabla \times \mathbf{B}_h^{n+1}) - \mathbf{r}_V\rangle = \langle \mathbf{v}, \boldsymbol{\nu}_{res}\rangle &&\forall \mathbf{v} \in Nc_k^f(\Omega_p). \label{nu_res}
\end{align}
\end{proposition}
\begin{proof}
Note that in the magnetic field equation \eqref{B_eqn_V_mod}, we can rewrite
\begin{equation}
    \left\langle \nabla \times \mathbf{\Sigma}, \mathbf{B}_{h,0} \times \mathbf{V}_h^{n+1}(\mathbf{B}_{h,0}) \right\rangle = - \left\langle \mathbf{B}_{h,0} \times (\nabla \times \mathbf{\Sigma}), \mathbf{V}_h^{n+1}(\mathbf{B}_{h,0})\right\rangle. \label{cross_shuffle}
\end{equation}
We then add and subtract
\begin{equation*}
    \alpha \delta t \langle c_0 \nabla \times \mathbf{\Sigma}, \mathbf{V}_h^{n+1}(\mathbf{B}_{h,0}) \rangle,
\end{equation*}
to \eqref{B_eqn_V_mod}, which together with \eqref{cross_shuffle} results in
\begingroup
\addtolength{\jot}{2mm}
\begin{align}
    &\left\langle \mathbf{\Sigma}, \mathbf{B}_h^{n+1} \right\rangle - \alpha\delta t \left\langle \boldsymbol{\sigma}(\mathbf{B}_{h,0}), \mathbf{V}_h^{n+1}(\mathbf{B}_{h,0}) \right\rangle \nonumber \\
    &\hspace{5mm}+ \alpha \delta t \langle c_0 \nabla \times \mathbf{\Sigma}, \mathbf{V}_h^{n+1}(\mathbf{B}_{h,0}) \rangle = \langle \mathbf{\Sigma}, \mathbf{r}_B\rangle & \forall \mathbf{\Sigma} \in Nc_k^e(\Omega_p), \label{B_eqn_proof}
\end{align}
\endgroup
where
\begin{equation*}
    \boldsymbol{\sigma}(\mathbf{B}_{h,0}) = \mathbf{B}_{h,0} \times (\nabla \times \mathbf{\Sigma}) + c_0 \nabla \times \mathbf{\Sigma} \in \mathbb{V}_{V,0},
\end{equation*}
which is a function of $\mathbb{V}_{V,0}$ since by FEEC, the curl maps from $Nc_k^e$ to $Nc_k^f$. Accordingly, we can set $\mathbf{w}(\mathbf{B}_{h,0}) = \boldsymbol{\sigma}(\mathbf{B}_{h,0})$ in the velocity equation \eqref{V_eqn_V_mod}, which implies
\begin{equation}
    \left\langle \boldsymbol{\sigma}(\mathbf{B}_{h,0}), \mathbf{V}_h^{n+1}(\mathbf{B}_{h,0}) \right\rangle = - \alpha\delta t \left\langle \boldsymbol{\sigma}(\mathbf{B}_{h,0}), \mathbf{B}_{h,0} \times (\nabla \times \mathbf{B}_h^{n+1})\right\rangle + \langle \boldsymbol{\sigma}(\mathbf{B}_{h,0}), \mathbf{r}_V\rangle. \label{V_subst_proof}
\end{equation}
Finally, substituting into \eqref{B_eqn_proof}, we obtain
\begingroup
\addtolength{\jot}{2mm}
\begin{align}
    &\left\langle \mathbf{\Sigma}, \mathbf{B}_h^{n+1} \right\rangle - \alpha\delta t \left(-\alpha\delta t \left\langle \boldsymbol{\sigma}(\mathbf{B}_{h,0}), \mathbf{B}_{h,0} \times (\nabla \times \mathbf{B}_h^{n+1})\right\rangle + \langle \boldsymbol{\sigma}(\mathbf{B}_{h,0}), \mathbf{r}_V\rangle\right) \nonumber \\
    &\hspace{5mm}+ \alpha \delta t \langle c_0 \nabla \times \mathbf{\Sigma}, \mathbf{V}_h^{n+1}(\mathbf{B}_{h,0}) \rangle = R_B(\mathbf{\Sigma}) & \forall \mathbf{\Sigma} \in Nc_k^e(\Omega_p), \label{B_eqn_proof_subst}
\end{align}
\endgroup
and by expanding the expressions $\boldsymbol{\sigma}(\mathbf{B}_{h,0})$ as well as moving terms appropriately, we arrive at
\begingroup
\addtolength{\jot}{2mm}
\begin{align*}
    &\left\langle \mathbf{\Sigma}, \mathbf{B}_h^{n+1} \right\rangle + \alpha^2\delta t^2 \left\langle \mathbf{B}_{h,0} \times (\nabla \times \mathbf{\Sigma}), \mathbf{B}_{h,0} \times (\nabla \times \mathbf{B}_h^{n+1})\right\rangle \nonumber \\
    &\hspace{3mm}+ \alpha \delta t c_0 \left(\langle \nabla \!\times\! \mathbf{\Sigma}, \mathbf{V}_h^{n+1}(\mathbf{B}_{h,0}) \rangle \!+\! \alpha\delta t \left\langle \nabla \times \mathbf{\Sigma}, \mathbf{B}_{h,0} \!\times\! (\nabla \!\times\! \mathbf{B}_h^{n+1})\right\rangle \!-\! \langle \nabla \!\times\! \mathbf{\Sigma}, \mathbf{r}_V\rangle\right) \nonumber \\
    &\hspace{5mm}= R_B(\mathbf{\Sigma}) -\alpha \delta t \langle \nabla \times \mathbf{\Sigma}, \mathbf{B}_{h,0} \times \mathbf{r}_V\rangle & \forall \mathbf{\Sigma} \in Nc_k^e(\Omega_p).
\end{align*}
\endgroup
Equation \eqref{B_Schur_eqn_V_mod} then follows directly from defining $\boldsymbol{\nu}_{res}$ as in \eqref{nu_res}, noting again that for any $\mathbf{\Sigma} \in Nc_k^e(\Omega_p)$, we have $\nabla \times \mathbf{\Sigma} \in Nc_k^f(\Omega_p)$.
\end{proof}
Equation \eqref{B_Schur_eqn_V_mod} is identical to the magnetic field equation \eqref{Schur_B_project} we considered in our Schur complement preconditioning discussion, except that, as targeted through our modified velocity space, the velocity projection no longer appears. Instead, due to the regularizing shift with factor $c_0$, we have an additional contribution related to a curl applied weakly to $\alpha \delta t c_0 \boldsymbol{\nu}_{res}$. $\boldsymbol{\nu}_{res}$ in turn corresponds to the residual of the velocity equation with test functions $\mathbf{w} = c_0\mathbf{v} \in Nc_k^f(\Omega)$ instead of $\mathbf{w} = \mathbf{B}_{h,0}\times \mathbf{v} + c_0 \mathbf{v} \in \mathbb{V}_{V,0}(\Omega)$. In particular, the term is strongly consistent in the sense that before discretization (in space and time), the strong velocity residual evaluates to zero. In practice, when preconditioning the Schur complement in $\mathbf{B}$ in our full set of MHD equations, we disregard the $\boldsymbol{\nu}_{res}$ term and use the left-hand side of \eqref{B_Schur_eqn_V_mod} to construct an approximate Schur complement.

\begin{remark}(Density factor)
 Finally, we note that for the full resistive MHD equations \eqref{3D_MHD_non_discr}, the velocity field's time derivative contains a factor of $n_0$ (after linearization), for some known density field $n_0$. Accordingly, in this case, the corresponding weak equation in $\mathbf{B}$ \eqref{B_wave_eqn} contains a factor of $1/n_0$ in one of the cross-products. For non-constant $n_0$, the proof of Proposition \ref{prop_Schur} -- which constructs a discrete version of \eqref{B_wave_eqn} -- no longer holds true, since in this case, the left-hand side term of \eqref{V_subst_proof} reads $\left\langle n_0 \boldsymbol{\sigma}(\mathbf{B}_{h,0}), \mathbf{V}_h^{n+1}(\mathbf{B}_{h,0}) \right\rangle$ and can no longer be directly substituted into equation \eqref{B_eqn_proof}. The substitution would still work if we considered a discrete velocity equation based on a strong one \eqref{momentum_eqn_cts} that has been divided by $n$; however, in this case, the discretization can be shown to contain an additional energetic inconsistency (in that the velocity equation's substitution in \eqref{H_subst} in Appendix \ref{app_structure} below no longer holds true). In practice, however, we find that an approximate Schur complement based on the left-hand side of \eqref{B_Schur_eqn_V_mod} and containing a non-constant factor of $n_0^{-1}$ still remains a good representation of the true Schur complement.
\end{remark}

\subsubsection{Split time discretization}

Above we developed a solver for the isolated $\mathbf{V}$-$\mathbf{B}$ system within the plasma domain $\Omega_p$; next we need to incorporate this with full resistive MHD dynamics. However, a fully implicit coupling makes for a much less clear Schur complement and associated approximation than in the isolated $\mathbf{V}$-$\mathbf{B}$ system. Specifically, the $\{\delta \mathbf{V},\delta n,\delta T\}$ $3 \times 3$ block subsystem that we eliminate has non-trivial off-diagonal structure including sound wave coupling terms that relate the velocity to the density and temperature fields; even when $\beta$ is small, this coupling cannot be ignored for adequately fine meshes in tokamak domains.

To circumvent this, we use an implicit-implicit partitioned integration approach for resistive MHD developed in our companion paper \cite{krzysik2026imim}. For completeness, we briefly review the method here as well. Due to the modified velocity space and nonlinear interaction of different variables, we need to use a \emph{nonlinear partitioning} as proposed in \cite{buvoli2025new,tran2025order}. Specifically, we map the right-hand side of an ODE $\widehat{F}(y)$ to a function of two arguments:
\begin{equation}
\pp{y}{t} = \widehat{F}(y) \coloneqq F\left(y^{[1]}, y^{[2]}\right), \label{IMIM_scheme}
\end{equation}
were $F\left(y,y\right) = \widehat{F}(y)$ when $y = y^{[1]}=y^{[2]}$. In this context, we use a second order implicit-implicit (IMIM) nonlinearly partitioned Runge--Kutta method  \cite{buvoli2025new} of the form
\begin{equation}
\begin{aligned}
&Y^{(0)} = y^n, &\hspace{15mm} Y^{(2)} = y^n + \frac{\delta t}{2} F(Y^{(1)}, Y^{(2)}), \\ 
&Y^{(1)} = y^n + \frac{\delta t}{2} F(Y^{(1)}, Y^{(0)}), &\hspace{15mm} y^{n+1} = y^n + \delta t F(Y^{(1)}, Y^{(2)}),
\end{aligned}
\label{ynp1}
\end{equation}
for time step $\delta t$, stage vectors $Y^{(0)}$, $Y^{(1)}$ and $Y^{(2)}$,  and current as well as next time level states $y^n$, $y^{n+1}$, respectively. Note that $y^{n+1}$ can be expressed through $y^n$ and $Y^{(2)}$ by substituting for $F(Y^{(1)}, Y^{(2)})$, so that in practice, this scheme involves the computation of two implicit stages to obtain $Y^{(1)}$ and $Y^{(2)}$. In view of our linearized system \eqref{MHD_system_full}, we then split the MHD equations \eqref{3D_MHD_non_discr} according to arguments in $F(y^{[1]}, y^{[2]})$ such that
\begin{align}
\begin{split}
& y^{[1]}\colon \text{Sound wave coupling, viscosity, interior penalty terms,}\\
& y^{[2]}\colon \text{Magnetic wave coupling, heat flux.}\label{IMIM_split}
\end{split}
\end{align}
Detailed derivations and numerical experiments regarding the IMIM partitioning, implementation, and interaction with the modified velocity space can be found in our companion paper \cite{krzysik2026imim}. For our purposes, following the split, the $4 \times 4$ block system \eqref{MHD_system_full} decouples into two implicit systems, with the first one given by
\begin{equation}
\begin{pmatrix}
  \frac{1}{\alpha\delta t}M^V_{n_0} \!+\! L_{\kappa_V} \!+\! L_{Re} & \beta G_{T_0} & \beta G_{n_0} & 0\\
  - G_{n_0}^T & \frac{1}{\alpha\delta t}M^n \!+\! L_{\kappa_n}  & 0 & 0 \\
  -G_{n_0T_0}^T & 0 & \tfrac{1}{\alpha\delta t(\gamma-1)}M^T_{n_0} \!+\! L_{\kappa_T} & 0 \\
  0 & 0 & 0 & \frac{1}{\alpha\delta t}M^B \!+\! L_{\kappa_B}
\end{pmatrix}
\!\!
\begin{pmatrix}
  \delta \mathbf{V}^{(1)}\\
  \delta n^{(1)} \\
  \delta T^{(1)} \\
  \delta \mathbf{B}^{(1)}
\end{pmatrix}
= R^{(0)}, \label{MHD_system_sound}
\end{equation}
for thermohydrodynamic terms and magnetic field stabilization. The updates $\delta x^{(1)}$, for $x\in \{\mathbf{V}, n, T, \mathbf{B}\}$, correspond to the stage vector $Y^{(1)}$, and the right-hand side residual $R^{(0)}$ is a function of $y^n$ and the known stage vector $Y^{(0)}$. Further, for \eqref{ynp1}, the time stepping coefficient is given by $\alpha = 1/2$. The second system is given by
\begin{equation}
\begin{pmatrix}
  \frac{1}{\alpha\delta t}M^V_{n_0} & 0 & 0 & C_{B_0} \\
  0 & \frac{1}{\alpha\delta t}M^n & 0 & 0 \\
  0 & 0 & \tfrac{1}{\alpha\delta t(\gamma-1)}M^T_{n_0} + Q & 0 \\
  -C_{B_0}^T & 0 & 0 & \frac{1}{\alpha\delta t}M^B + L_S
\end{pmatrix}
\!\!
\begin{pmatrix}
  \delta \mathbf{V}^{(2)}\\
  \delta n^{(2)} \\
  \delta T^{(2)} \\
  \delta \mathbf{B}^{(2)}
\end{pmatrix}
= R^{(1)}, \label{MHD_system_Alfven}
\end{equation}
containing primarily the stiff, fast-scale magnetic wave and heat flux terms. The updates and residual are defined analogously to the first system. We then observe that in \eqref{MHD_system_sound}, the magnetic field stabilization decouples and can be solved for separately, followed by a solve for the $3\times3$ thermodynamic system without the heat flux.

Finally, in \eqref{MHD_system_Alfven}, the heat flux problem in $\delta T$ and the magnetic wave coupling in $\delta \mathbf{V}$-$\delta \mathbf{B}$ decouple and can be solved for separately, while the problem for $\delta n$ becomes trivial. Eliminating $\{\delta n,\delta  T,\delta \mathbf{V}\}$ yields a simple Schur complement in $\delta\mathbf{B}$ of the form
\begin{equation}
    \tfrac{1}{\alpha \delta t}M_B + \alpha \delta tC_{B_0}^T (M_{n_0}^V)^{-1} C_{B_0} + L_S.
\end{equation}
Following the previous discussion in this section, with the velocity mass matrix inverse corresponding to the velocity space projection, and scaling through by a constant $\alpha \delta t$, this operator can then be preconditioned based on the weak variational form
\begingroup
\addtolength{\jot}{2mm}
\begin{align}
    s_{B,n_0}' = &\left\langle \mathbf{\Sigma}, \delta \mathbf{B} \right\rangle_{\!\Omega} + \alpha^2\delta t^2 \left\langle \nabla \times \mathbf{\Sigma}, \tfrac{1}{n_0}\boldsymbol{\Omega}_{h,0}\nabla \times \delta \mathbf{B} \right\rangle + \alpha \delta t \left\langle \nabla \times \mathbf{\Sigma}, S^{-1}\nabla \times \delta \mathbf{B} \right\rangle_{\!\Omega} & \forall \mathbf{\Sigma} \in Nc_k^e(\Omega), \label{Schur_B_weak}
\end{align}
\endgroup
together with AMS. Note that since the full system includes resistivity, which is also evolved in the vacuum vessel-wall region $\Omega_{vw}$, we have included inner products $\langle.,.\rangle_{\Omega}$ here.

\begin{remark}[Schur complement in fully discrete MHD scheme] \label{remark_Schur_nonl_MHD_IMIM}
In Section \ref{V_B_coupling_discr}, we derived a Schur-complement based preconditioning strategy for a discretization of the linearized magnetic wave equations \eqref{BV_system}. For the fuly discretized, nonlinear MHD equations, additional terms appear from the linearization (within an iterative Newton method) as well as from the IMIM time discretization. The former terms generally do not contribute as much to the discrete linear system's overall stiffness, since they either do not involve differentials of unknowns or are multiplied by very small quantities. Similarly, for the latter terms, it can be shown that either the same holds true, or they appear as known right-hand side terms.
\end{remark}

%
%
%
%
%
\subsection{Full MHD scheme} \label{sec_full_scheme}
Having discussed the solver strategy targeted at the stiff magnetic wave coupling, we next present the process for the full MHD scheme, based on the spatial discretization \eqref{3D_MHD_discr} together with the velocity space $\mathbb{V}_V(\Omega_p) = \mathbb{V}_V[\mathbf{B}_h](\Omega_p)$ from \eqref{V_space_mod_0}, as well as the IMIM split-time discretization \eqref{ynp1} with our split choice \eqref{IMIM_split}. For further details on IMIM discretizations, including on the time derivative applied to $\mathbf{V}_h = \mathbf{B}_h\times\boldsymbol{\mathcal{U}}_h + c_0 \boldsymbol{\mathcal{U}}_h$ in the momentum equation, see \cite{krzysik2026imim}.

As discussed, the fully discretized equations include two implicit stages. Given the current prognostic time level fields $(n_h^n, \boldsymbol{\mathcal{U}}_h^n, T_h^n, \mathbf{B}_h^n)$, the first stage solves for prognostic fields $(n_h^{(1)}, \boldsymbol{\mathcal{U}}_h^{(1)}, T_h^{(1)}, \mathbf{B}_h^{(1)})$ implicitly with respect to sound wave, viscosity, and interior penalty terms such that
\begingroup
\allowdisplaybreaks
\addtolength{\jot}{3mm}
\begin{subequations} \label{3D_MHD_discr_stage_I}
\begin{align}
&\left\langle \chi, \tfrac{2}{\delta t}\big(n_h^{(1)} - n_h^n\big)\right\rangle - \left\langle \nabla \chi, \mathbf{V}_h^{(1)} n_h^{(1)} \right\rangle + L_{\kappa_n}\big(\chi, n_h^{(1)}\big) = 0& \forall \chi \in Q_k(\Omega_p),\\
&\left\langle n_h^n \mathbf{w}\big(\mathbf{B}^{(1)}_h, \mathbf{v}\big), \tfrac{2}{\delta t}\big(\mathbf{V}^{(1)}_h \!-\! \mathbf{V}_h^n\big) +\boldsymbol{\omega}_h\big(\mathbf{V}_h^{(1)}\big) \!\times\! \mathbf{V}_h^{(1)} \!+\! \tfrac{1}{2}\nabla P_h\big(\mathbf{V}_h^{(1)}\big) \right\rangle \nonumber \\
&\hspace{2mm}+ \left\langle \mathbf{w}(\mathbf{B}_h^{(1)}, \mathbf{v}), \beta \nabla \big(n_h^{(1)} T_h^{(1)}\big) + \mathbf{B}_h^n \times \big(\nabla \times \mathbf{B}_h^n\big) \right\rangle \nonumber \\
& \hspace{2mm} + L\big(\mathbf{w}\big(\mathbf{B}_h^{(1)}, \mathbf{v}\big), \mathbf{V}_h^{(1)}\big) + L_{\kappa_V}\big(\mathbf{v}, \boldsymbol{\mathcal{U}}_h^{(1)}\big) =0 & \forall \mathbf{v} \in \mathring{Nc}_k^f(\Omega_p), \label{V_eqn_discr_stage_I} \\
&\left\langle \frac{n_h^n \eta}{\gamma - 1}, \tfrac{2}{\delta t}\big(T^{(1)}_h - T^n_h\big) + \mathbf{V}_h^{(1)}\cdot \nabla T_h^{(1)}\right\rangle - \left\langle \nabla(\eta  n_h^{(1)} T_h^{(1)}), \mathbf{V}_h^{(1)} \right\rangle \nonumber \\
&\hspace{2mm}+ \left\langle \mathbf{b}_h^n \cdot \nabla \eta, Pe_\Delta^{-1} \zeta_h\big(\mathbf{b}_h^n, T_h^n\big) \right\rangle \!+\! \left\langle \nabla \eta, Pe_\perp^{-1} \nabla T_h^n \right\rangle \!+\! L_{\kappa_T}\big(\eta, n_h^{(1)}, T_h^{(1)}\big) = 0 & \forall \eta \in \mathring{Q}_k(\Omega_p), \\
&\left\langle \mathbf{\Sigma}, \tfrac{2}{\delta t}\big(\mathbf{B}_h^{(1)} - \mathbf{B}^n_h\big) \right\rangle_{\!\Omega} - \left\langle \nabla \times \mathbf{\Sigma}, \mathbf{V}_h^n \times \mathbf{B}_h^n\right\rangle \nonumber \\
&\hspace{2mm}+ \left\langle \nabla \times \mathbf{\Sigma}, S^{-1}\big(T_h^n\big) \nabla \times \mathbf{B}_h^n\right\rangle_{\Omega} + L_{\kappa_B}\big(\mathbf{\Sigma}, \mathbf{B}_h^{(1)}\big) = 0&\forall \mathbf{\Sigma} \in Nc_k^e(\Omega), \label{B_eqn_discr_stage_I}
\end{align}
\end{subequations}
\endgroup
where the auxiliary variables $\boldsymbol{\omega}_h$, $P_h$, and $\zeta_h$ are defined as in \eqref{3D_MHD_discr} with the corresponding time-discrete arguments. Further, we used expressions
\begin{equation}
\mathbf{w}(\mathbf{B}_h^{(1)}, \mathbf{v}) = \mathbf{B}_h^{(1)} \times \mathbf{v} + c_0 \mathbf{v}, \;\;\;\; \mathbf{V}_h^{(1)} = \mathbf{B}_h^{(1)} \times \boldsymbol{\mathcal{U}}^{(1)}_h + c_0 \boldsymbol{\mathcal{U}}_h^{(1)}, \;\;\;\; \mathbf{V}_h^n = \mathbf{B}_h^n \times \boldsymbol{\mathcal{U}}^n_h + c_0 \boldsymbol{\mathcal{U}}_h^n,
\end{equation}
as well as $\mathbf{b}_h^n = \mathbf{B}_h^n/|\mathbf{B}_h^n|$. Before moving on to the second stage, we note that in the velocity interior penalty term $L_{\kappa_V}$ (cf. \eqref{IP_V}), we set $f$ equal to the inverse relation \eqref{V_mod_inv}, that is, we penalize $\boldsymbol{\mathcal{U}}_h$ instead of $\mathbf{V}_h$, which we found to lead to better results.

In the second stage, given the current prognostic time level and first stage fields $(n_h^n, \boldsymbol{\mathcal{U}}_h^n, T_h^n, \mathbf{B}_h^n)$, $(n_h^{(1)}, \boldsymbol{\mathcal{U}}_h^{(1)}, T_h^{(1)}, \mathbf{B}_h^{(1)})$, respectively, we solve for $(n_h^{(2)}, \boldsymbol{\mathcal{U}}_h^{(2)}, T_h^{(2)}, \mathbf{B}_h^{(2)})$ implicitly in the magnetic wave coupling and heat flux such that
\begingroup
\allowdisplaybreaks
\addtolength{\jot}{3mm}
\begin{subequations} \label{3D_MHD_discr_stage_II}
\begin{align}
&\left\langle \chi, \tfrac{2}{\delta t} \big(n_h^{(2)} - n_h^n\big)\right\rangle - \left\langle \nabla \chi, \mathbf{V}_h^{(1)} n_h^{(1)} \right\rangle + L_{\kappa_n}\big(\chi, n_h^{(1)}\big) = 0& \forall \chi \in Q_k(\Omega_p),\\
&\left\langle n_h^{(2)} \mathbf{w}\big(\mathbf{B}^{(1)}_h, \mathbf{v}\big), \tfrac{2}{\delta t}\big(\mathbf{V}^{(2)}_h \!-\! \mathbf{V}_h^n\big) \!+\! g_{[B,\mathcal{U}]} \!+\!  \boldsymbol{\omega}_h\big(\mathbf{V}_h^{(1)}\big) \!\times\! \mathbf{V}_h^{(1)} \!\!+\! \tfrac{1}{2}\nabla P_h\big(\mathbf{V}_h^{(1)}\big)\!\! \right\rangle \nonumber \\
&\hspace{2mm}+ \left\langle \mathbf{w}(\mathbf{B}_h^{(1)}, \mathbf{v}), \beta \nabla \big(n_h^{(1)} T_h^{(1)}\big) + \mathbf{B}_h^{(2)} \times \big(\nabla \times \mathbf{B}_h^{(2)}\big) \right\rangle \nonumber \\
& \hspace{2mm} + L\big(\mathbf{w}\big(\mathbf{B}_h^{(1)}, \mathbf{v}\big), \mathbf{V}_h^{(1)}\big) + L_{\kappa_V}\big(\mathbf{v}, \boldsymbol{\mathcal{U}}_h^{(1)}\big) =0 & \forall \mathbf{v} \in \mathring{Nc}_k^f(\Omega_p), \label{3D_MHD_discr_V_eqn_stage_II}\\
&\left\langle \frac{n_h^{(2)} \eta}{\gamma - 1}, \tfrac{2}{\delta t}\big(T^{(2)}_h - T^n_h\big) + \mathbf{V}_h^{(1)}\cdot \nabla T_h^{(1)}\right\rangle - \left\langle \nabla(\eta  n_h^{(1)} T_h^{(1)}), \mathbf{V}_h^{(1)} \right\rangle \nonumber \\
&\hspace{2mm}\!+ \!\!\left\langle \mathbf{b}_h^{(2)} \!\!\cdot\!\! \nabla \eta, Pe_\Delta^{\!-1} \zeta_h\big(\mathbf{b}_h^{(2)}, T_h^{(2)}\big)\!\! \right\rangle \!+\! \left\langle\! \nabla \eta,\! Pe_\perp^{-1} \nabla T_h^{(2)}\! \right\rangle \!+\!\! L_{\kappa_T}\big(\eta, n_h^{(1)}\!, T_h^{(1)}\big) \!=\! 0 & \forall \eta \in \mathring{Q}_k(\Omega_p), \label{3D_MHD_discr_T_eqn_stage_II} \\
&\left\langle \mathbf{\Sigma}, \tfrac{2}{\delta t}\big(\mathbf{B}_h^{(2)} - \mathbf{B}^n_h\big) \right\rangle_\Omega  - \left\langle \nabla \times \mathbf{\Sigma}, \mathbf{V}_h^{(2)} \times \mathbf{B}_h^{(2)} \right\rangle\nonumber \\
&\hspace{2mm}+ \left\langle \nabla \times \mathbf{\Sigma}, S^{-1}\big(T_h^{(2)}\big) \nabla \times \mathbf{B}_h^{(2)}\right\rangle_\Omega + L_{\kappa_B}\big(\mathbf{\Sigma}, \mathbf{B}_h^{(1)}\big) = 0&\forall \mathbf{\Sigma} \in Nc_k^e(\Omega), \label{3D_MHD_discr_B_eqn_stage_II}
\end{align}
\end{subequations}
\endgroup
where $\mathbf{V}_h^{(2)}$ and $\mathbf{b}_h^{(2)}$ are defined analogously to $\mathbf{V}_h^{(1)}$ and $\mathbf{b}_h^{(1)}$, respectively. Additionally, the discrete time derivative of the $\mathbf{B}_h$-dependent expression $\mathbf{V}_h$ contains an additional term
\begin{equation}
g_{[B,\mathcal{U}]} = \tfrac{1}{2}\big(\delta \mathbf{B}_h^{(2)} \times \delta \boldsymbol{\mathcal{U}}_h^{(1)} + \delta \mathbf{B}_h^{(1)} \times \delta \boldsymbol{\mathcal{U}}_h^{(2)} \big),    
\end{equation}
which is small in practice (see \cite{krzysik2026imim}).

\subsubsection{Solver procedure}

We next outline the different solver procedures for the two fully discretized stages above, including the outer block preconditioning strategy and inner solve details. In particular, we note that we employ a Newton method with true Jacobians rather than a quasi-Newton one with approximate Jacobians akin to the ones we used to motivate our solver in the previous section, which we only introduced for illustration purposes. As mentioned in Remark \ref{remark_Schur_nonl_MHD_IMIM}, while this leads to a number of additional terms not included in our motivating linearized system \eqref{MHD_system_full}, these do not change our overall preconditioning strategy, nor our approximate Schur complement based on the varational form \eqref{Schur_B_weak}. A conceptual diagram of the overall procedure, including the relative linear and nonlinear solver tolerances used in the numerical results section, is given in Figure \ref{fig_IMIM_solves}.

\textbf{Stage 1.} For \eqref{3D_MHD_discr_stage_I}, we first note that as in the simplified linearized setup \eqref{MHD_system_sound}, the magnetic field equation \eqref{B_eqn_discr_stage_I} is decoupled from the remaining equations. We can therefore solve it separately first to obtain $\mathbf{B}_h^{(1)}$, which requires an inverse involving the interior penalty term $L_{\kappa_B}$. The system arising from the magnetic field equation is approximately inverted using the conjugate gradient method, with an additive Schwartz method based on extruded, vertex-centered star patches\footnote{The patches' extrusion is based on the underlying extruded meshes we use for our numerical tests: a straight extrusion of a 2D periodic square mesh to create a 3D box mesh, and a curved extrusion of a 2D vertical cross-section mesh to create a 3D tokamak mesh.} as a preconditioner. Additionally, the auxiliary variable $\zeta_h(\mathbf{b}_h^n, T_h^n)$ can also be solved for separately first, inverting the corresponding DG mass matrix.

This leaves the (nonlinear) coupling in $n_h^{(1)}$, $\boldsymbol{\mathcal{U}}_h^{(1)}$, $T_h^{(1)}$ as well as the auxiliary variables $\boldsymbol{\omega}_h$, $P_h$ depending on the field $\boldsymbol{\mathcal{U}}_h^{(1)}$ to be solved for. To motivate our block preconditioning setup, we note that the $5 \times 5$ system of equations corresponding to the Jacobian reads
\begingroup
\setlength{\arraycolsep}{4.5pt}
\begin{equation}
\begin{pmatrix}
  M^P & 0 & m_{13} & 0 & 0\\
  0 & M^\omega & C_\mathcal{U}^\top & 0 & 0\\
  \tfrac{1}{2}G & m_{32} & \frac{1}{\alpha\delta t}M^\mathcal{U}_{n_0} \!+\! m_{33} \!+\! L_{\kappa_V} \!+\! L_{Re} & \beta G_{T_0} & \beta G_{n_0}\\
  0 & 0 & - G_{n_0}^T & \frac{1}{\alpha\delta t}M^n \!+\! L_{\kappa_n} \!+\! m_{44}  & 0 \\
  0 & 0 & -G_{n_0T_0}^T \!+\! g_{53} & m_{54} & \tfrac{1}{\alpha\delta t(\gamma-1)}(M^T_{n_0} \!+\! m_{55}) \!+\! L_{\kappa_T} \\
\end{pmatrix}
\!\!\!
\begin{pmatrix}
    P_h \\
    \boldsymbol{\omega}_h \\ 
    \delta \boldsymbol{\mathcal{U}}^{(1)}\\
    \delta n^{(1)}\\
    \delta T^{(1)}
\end{pmatrix}.
\label{MHD_system_sound_full}
\end{equation}
\endgroup
The blocks are defined analogously to \eqref{MHD_system_sound}, except for additional components involving the cross product of the known field $\mathbf{B}_h^{(1)}$, which arise from solving for $\boldsymbol{\mathcal{U}}$ instead of $\mathbf{V}_h$. For instance, for $M^\mathcal{U}_{n_0}$ we have entries
\begin{align}
\{M^\mathcal{U}_{n_0}\big\}_{ij} = \left\langle \mathbf{B}_h^{(1)} \times \mathbf{v}_i + c_0 \mathbf{v}_i, n_h^n \left(\mathbf{B}_h^{(1)} \times \mathbf{v}_j + c_0 \mathbf{v}_j\right) \right\rangle &&\forall \mathbf{v}_i, \mathbf{v}_j \in \mathring{Nc}_k^f(\Omega_p),
\end{align}
which in particular is still symmetric. In addition, we have the blocks $\tfrac{1}{2}G$ and $C_\mathcal{U}^\top$, which correspond to discrete gradient and weak curl operations for the auxiliary variables $P_h$ and $\boldsymbol{\omega}_h$ related to velocity advection. The mass-like blocks $m_{13}$, $m_{32}$, and $m_{33}$ are Jacobian contributions arising from the linearization of the velocity advection terms, while $m_{44}$ and $m_{55}$ arise analogously from the density and temperature advection terms, respectively. The block $m_{54}$ results from the linearization of the sound-wave coupling. Finally, $g_{53}$ is an additional contribution arising from the temperature advection term and has the structure of a transposed discrete gradient.

We precondition the resulting Jacobian system of 5 unknowns using a block Gauss-Seidel decomposition to split off $\boldsymbol{\omega}_h$, $P_h$, followed by a Schur complement to eliminate $\boldsymbol{\mathcal{U}}_h^{(1)}$ in the remaining system in $n_h^{(1)}$, $\boldsymbol{\mathcal{U}}_h^{(1)}$, $T_h^{(1)}$. Specifically, we use a lower triangular factorization of the Schur complement preconditioner, approximating the velocity block inverse in the Schur complement by the inverse of its diagonal. To split the remaining fields $n_h^{(1)}$ and $T_h^{(1)}$, we then again apply a block Gauss-Seidel decomposition. Finally, on the outermost level of the $5\times5$ block system, we apply GMRES \cite{saad1986gmres}, and in each iteration along with this block preconditioning strategy, we apply a) a Jacobi (diagonal) preconditioner for the $\boldsymbol{\omega}_h$, $P_h$ mass matrices, and b) a single application of an algebraic multigrid (AMG) solver via Hypre (BoomerAMG) \cite{falgout2002hypre} for the velocity block and temperature as well as density Schur complements each. Note that such standard AMG methods will not be robust for the hydrodynamics including the div-grad viscosity and interior penalty terms, and here we use AMG for simplicity since this work focuses primarily on the velocity-magnetic field coupling.

\begin{figure}[ht]
\begin{center}
\includegraphics[width=1.0\textwidth]{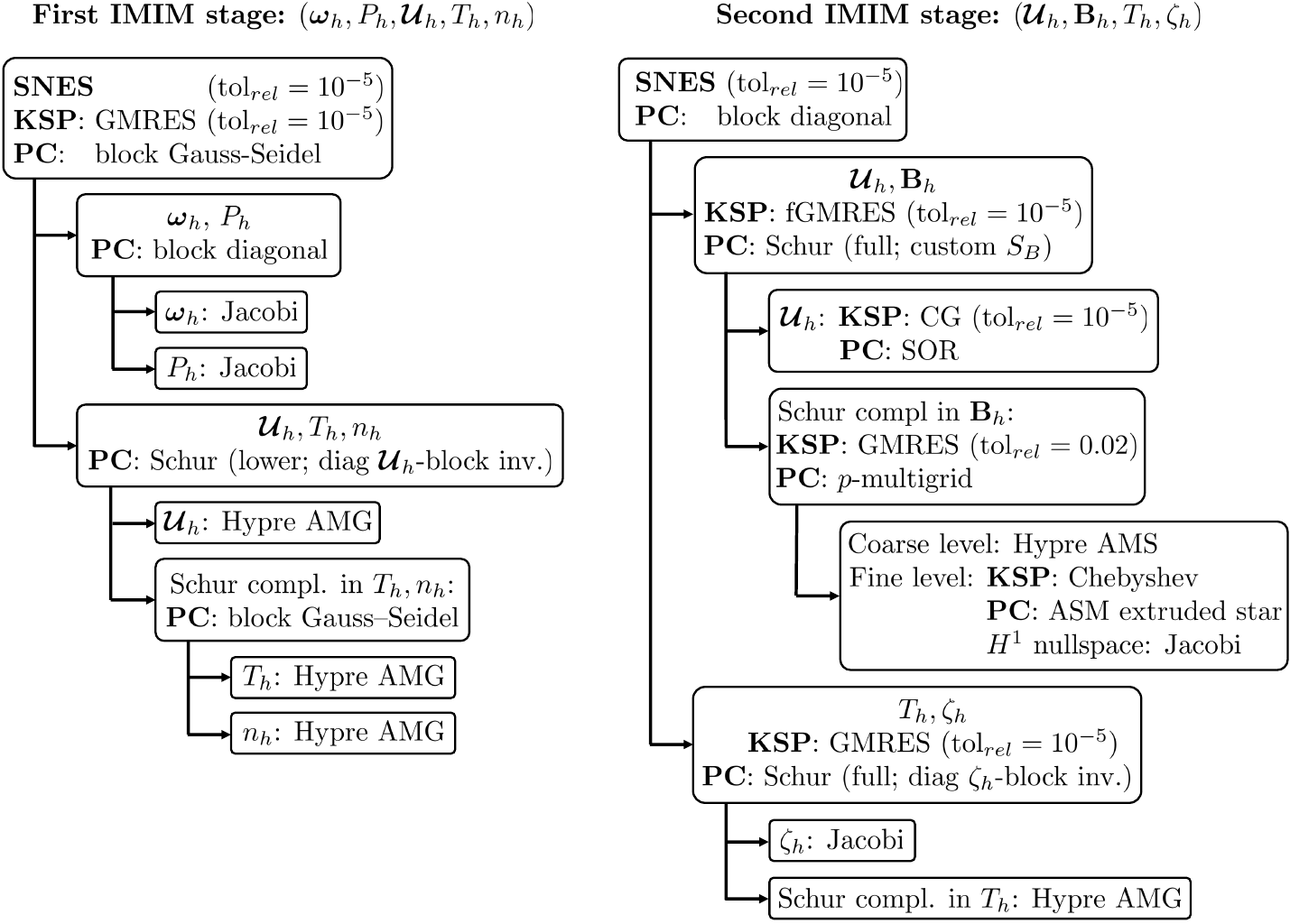}
\vspace{-5mm}
\caption{Diagrams depicting mixed solves for the first and second IMIM stages, centered around the thermohydrodynamic coupling (left) and the Alfvén wave coupling as well as heat flux terms (right), respectively. Here, we follow the naming convention of PETSc \cite{balay2019petsc}, where SNES, KSP, PC correspond to the nonlinear Newton solver, Krylov subspace method and preconditioner, respectively. $\boldsymbol{\omega}_h$ and $P_h$ are auxiliary variables corresponding to velocity advection, and $\boldsymbol{\mathcal{U}}_h$ is the prognostic velocity field we solve for, occurring in the overall velocity field expression $\mathbf{V}_h = \mathbf{B}_h \times \boldsymbol{\mathcal{U}}_h + c_0 \boldsymbol{\mathcal{U}}_h$. $\zeta_h$ is an auxiliary variable corresponding to the temperature's directional gradient. The custom approximation to the varitional form $s_B$ for the Schur complement in the magnetic field is given by $s_{B,n_0}'$ defined in \eqref{Schur_B_weak}. $\mathbf{B}_h$ in the first and $n_h$ in the second IMIM stages, respectively, are decoupled and solved for separately (using the same relative tolerances as for the outer linear mixed solves).} \label{fig_IMIM_solves}
\end{center}
\end{figure}

\textbf{Stage 2.} For the second stage, we first note that the density equation is decoupled, and we solve for $n_h^{(2)}$ using the conjugate gradient method preconditioned by SOR. The remaining equations form a system in $\boldsymbol{\mathcal{U}}_h^{(2)}$, $T_h^{(2)}$, $\mathbf{B}_h^{(2)}$ together with the auxiliary variable $\zeta_h$ (while the values for the auxiliary variables $\boldsymbol{\omega}_h$, $P_h$ can be reused from the Stage 1 solve), which reads
\begin{equation}
\begin{pmatrix}
  \frac{1}{\alpha\delta t}M^\mathcal{U}_{n_0} \!+m_{11}\! \!+\! L_{\kappa_V} \!+\! L_{Re}& 0 & 0 & C_{B_0} + m_{14}  \\
  0 & M^\zeta  & -G_b & m_{24}\\
  0 & G_b^\top & \!\!\!\!\tfrac{1}{\alpha\delta t(\gamma-1)}M^T_{n_0} \!+\! L_{\kappa_T} \!+\! L_T & m_{34}\\
  -C_{B_0}^\top & 0 & m_{43} & \!\!\!\!\frac{1}{\alpha\delta t}M^B \!\! + m_{44} \!+\! L_S
\end{pmatrix}
\!\!
\begin{pmatrix}
    \delta \boldsymbol{\mathcal{{U}}}_h^{(2)}\\
    \zeta_h\\
    \delta T_h^{(2)}\\
    \delta \mathbf{B}_h^{(2)}
\end{pmatrix}. \label{MHD_system_Alfven_full}
\end{equation}
$L_T$ denotes the isotropic heat flux, and $G_b$ corresponds to the directional gradient with respect to $\mathbf{b}_h$. As before, the terms $m_{ij}$ are mass-type Jacobian contributions arising from the linearization. Note that unlike the approximate Jacobian \eqref{MHD_system_Alfven}, the true Jacobian fully couples the temperature and magnetic field updates through $m_{24}$, $m_{34}$ and $m_{43}$, that is through $\mathbf{b}_h^{(2)}$ in the anisotropic heat flux in \eqref{3D_MHD_discr_T_eqn_stage_II} as well as $T_h^{(2)}$ in the resistivity in \eqref{3D_MHD_discr_B_eqn_stage_II}. In practice, for fully nonlinear 3D tokamak simulations, we consider time steps in the order of $\mathcal{O}(1)$ (with $\delta t =1$ corresponding to one Alfvén wave time unit in our non-dimensionalization) and $S$ scaled by at least $10^4$ (which realistically should be even larger), and the coupling due to the temperature-dependent resistivity is therefore very weak. Consequently, we decouple the Alfvén wave and heat flux solves in our preconditioner for \eqref{MHD_system_Alfven_full} (but not in the residuals for the overall Newton method). For the mixed heat flux solve in $(\zeta_h, T^{(2)})$, we then consider GMRES preconditioned by an approximate LDU factorization, with a Schur complement in $T_h^{(2)}$ and a single application of a Jacobi preconditioner for the $\zeta_h$ mass matrix. For the Schur complement in $T_h^{(2)}$, the inverse of the $\zeta_h$ mass matrix is approximated by the inverse of its diagonal, and the resulting approximate Schur complement is preconditioned using a single application of algebraic multigrid (BoomerAMG) in each iteration.

Finally, for the Alfvén wave coupling we apply flexible GMRES \cite{saad1993flexible} preconditioned by an approximate LDU factorization with an approximate Schur complement in (the Newton update for) $\mathbf{B}_h^{(2)}$. Within this factorization, for the velocity block, consisting of a symmetric mass matrix using test and trial functions from the modified function space $\mathring{\mathbb{V}}_V\big[\mathbf{B}_h^{(1)}\big](\Omega_p)$, we use the conjugate gradient method preconditioned by SOR. Further, we apply GMRES for the Schur complement in $\mathbf{B}_h^{(2)}$, with a maximum of 20 iterations per outer fGMRES iteration, which we precondition based on the variational formulation \eqref{Schur_B_weak} (using $n_0 = n_h^n$, $\mathbf{B}_{h,0} = \mathbf{B}_h^n$, and $S = S(T_h^n)$). In each GMRES iteration, we then apply a $p$-multigrid method using AMS on the (lowest-order) $p$-coarse level \cite{kolev2009parallel}. On the $p$-fine levels, we employ a Chebyshev-accelerated smoothing scheme based on an AMS-type preconditioner. The curl--curl relaxation is performed using a patch-based additive Schwarz method on extruded vertex patches. Finally, the auxiliary $H^1$ problem is approximated using a Jacobi (diagonal) method, noting that in practice, we found this to be sufficient within the wider context of the $p$-multigrid scheme.

\begin{remark}(Heat flux solve)
For the mixed solve in $T_h^{(2)}$-$\zeta_h\big(\mathbf{b}_h^{(2)}, T_h^{(2)}\big)$, we note that efficient solvers for non-mesh aligned anisotropic heat flux is a challenging and ongoing research topic (see e.g., \cite{southworth2026algebraic,wimmer2024fast}). Since in this work we focus on the Alfvén wave coupling, for simplicity, here we use an artificially low parallel thermal Péclet number \eqref{non_dim_coeff} together with the aforementioned Schur complement based approach with standard algebraic multigrid.
\end{remark}


%
%
\section{Numerical results} \label{sec_Numerical_results}
Having introduced our novel solver strategy and corresponding space and time discretization, we move on to presenting numerical results to test the method's order of accuracy and stability, as well as solver properties. The tests were performed using Firedrake~\cite{FiredrakeUserManual}, which heavily relies on PETSc \cite{balay2019petsc}. As discussed in Section \ref{sec_space_discr}, we consider hexahedral meshes with edge- and face-based (Raviart-Thomas-) Nédélec spaces for the magnetic field $\mathbf{B}_h$ and velocity field $\boldsymbol{\mathcal{U}}_h$, respectively. Further, we will consider spaces with polynomial degree $k=2$ (and $k-1=1$ for the auxiliary variable $\zeta_h$ DG space) throughout. Additionally, the continuous interior penalty terms in \eqref{CIP_terms} are set to\footnote{The main motivation for this $\delta t$-dependent penalty -- with $\delta t$ typically in the order of $\mathcal{O}(1)$ -- is that it allows the matrix associated with the magnetic field solve's left-hand side in \eqref{B_eqn_discr_stage_I} (multiplied by $\delta t$) to be assembled only once at the beginning of the simulation in our Firedrake-based implementation, even when variable time steps are used. This substantially reduces the computational cost associated with solving \eqref{B_eqn_discr_stage_I}. For consistency, $\kappa_n$ and $\kappa_T$ are chosen to be equal to $\kappa_B$.} $\kappa_n, \kappa_T, \kappa_B = 10^{-3}h_e/\delta t$, while the (discontinuous) interior penalty term for the velocity is set to $\kappa_V = h_e$. Finally, unless otherwise noted, the nonlinear as well as outer and inner linear solver tolerances are set as described in Figure \ref{fig_IMIM_solves}.

\subsection{Order of accuracy, discrete energy}

First, we test for our novel space discretization's accuracy as well as total energy evolution, noting that the scheme's zero-divergence property was already demonstrated in \cite{wimmer2024structure}. For this purpose, we consider a periodic box test case with domain $\Omega = [0, 1]^3$, together with a regular mesh resolution in $xyz$-coordinates given by $\{1/16, 1/16, 1/3\}\times2^{-a}$, for $a \in \{0, 1, 2\}$, as well as a fixed time step $\delta t = 0.02$, run up to $t_{max}=0.1$. The initial conditions are given by
\begin{align}
    &\mathbf{V}|_{t=0} \equiv \mathbf{0}, &\mathbf{B}|_{t=0} = B_z \mathbf{e}_{\mathbf{z}} + B_\phi g(r)\mathbf{e}_\phi, \;\;\;\;\; p|_{t=0} = p_b - \frac{B_\phi^2}{2\beta} \left(g(r)^2 + \int_0^r\tfrac{g(s)}{s}\;ds\right), \label{Box_IC}
\end{align}
for cylindrical coordinate angular and vertical  unit vectors $\mathbf{e}_\phi, \mathbf{e}_z$, respectively, radial coordinate $r$, and where $g(r) = r e^{-(r - r_0)^2/\sigma_0^2}$, for $r_0 = 0.25$, $\sigma_0 = 0.1$. Note that the integral involving $g$ can be computed analytically using the error function. Further, we set $B_z, B_\phi, p_b, \beta$ equal to $0.8, 1.5, 5, 0.02$, respectively, and the initial density and temperature expressions are defined as $n|_{t=0} = (p|_{t=0})^{0.3}$, $T|_{t=0} = (p|_{t=0})^{0.7}$. Finally, we set the Reynolds number to $Re = 10^4$, and skip the resistivity and heat flux terms. Together with \eqref{Box_IC}, this implies a scenario with a non-decaying equilibrium, such that the analytic solution is equal to the initial conditions. The simulation is run using nonlinear and linear solver tolerances of $10^{-10}$ throughout, and the resulting prognostic fields' $L^2$ errors are given in Figure \ref{fig_accuracy_energy}.

\begin{figure}[ht]
\begin{center}
\includegraphics[width=1.0\textwidth]{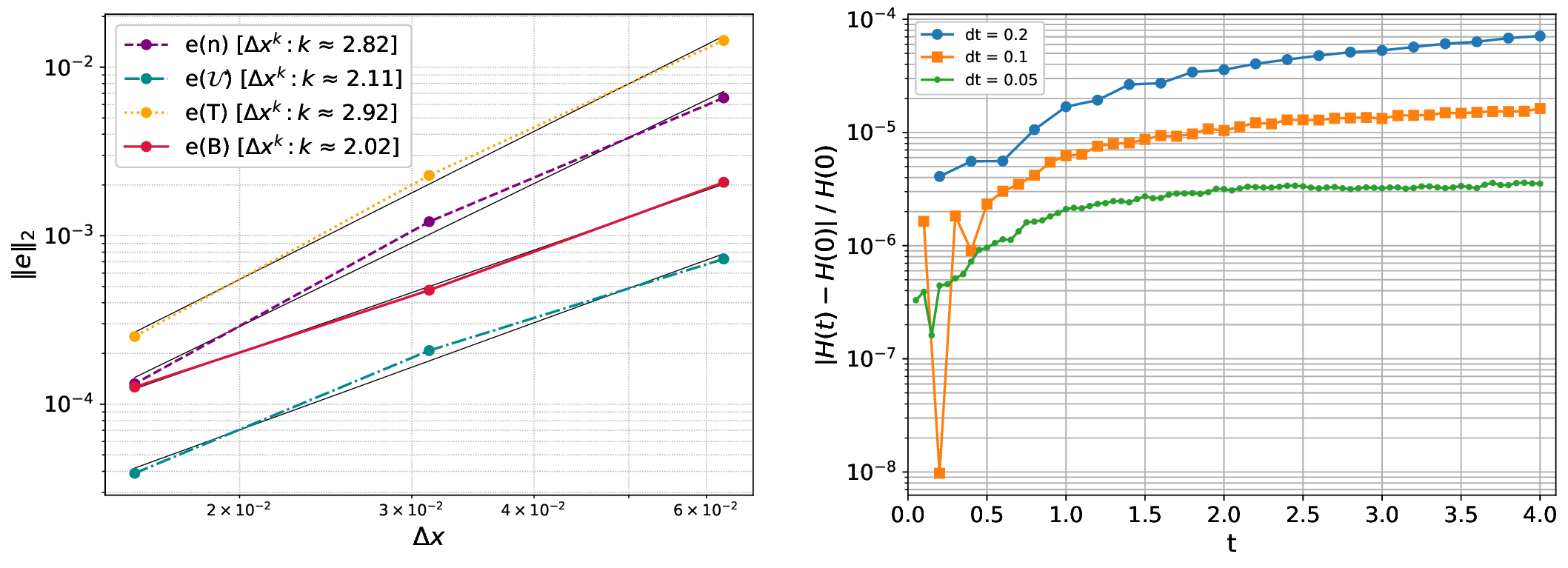}
\vspace{-5mm}
\caption{Plots based on test case \eqref{Box_IC}. Left: accuracy plots for the prognostic fields $n_h$, $\boldsymbol{\mathcal{U}}_h$, $T_h$, $\mathbf{B}_h$, with $e(n) = \|n_h - n|_{t=0}\|_2$, etc. Right: energy development over time.} \label{fig_accuracy_energy}
\end{center}
\end{figure}

We find that the spatial orders of accuracy are as expected, with near third order for the $CG_2(\Omega)$-fields $n_h$, $T_h$ and second order for the $Nc_2^e(\Omega)$, $Nc_2^f(\Omega)$ (Raviart-Thomas-) Nédélec spaces. Note that the inclusion of interior penalty terms may reduce the order of accuracy by a half \cite{burman2006edge}.

Finally, we consider the space discretization's energetic structural consistency by running the same test case at the lowest spatial resolution, and with three separate simulations with fixed time step $\delta t \in \{0.05, 0.1, 0.2\}$, up to $t_{max} = 4$ each. Further, we additionally skip the viscosity as well as density, velocity, and magnetic field interior penalty terms, such that the space-discretized semi-discrete total energy $H$ (see \eqref{H_final}) is conserved. The resulting total energy development -- as a function of time -- is depicted in Figure \ref{fig_accuracy_energy}, and we find that as expected, the energy error decreases as the time step is decreased (noting that the IMIM time scheme is not energy conserving). The simulations' last values are given by approximately $7.11 \times 10^{-5}$, $1.63 \times 10^{-5}$, and $3.53 \times 10^{-6}$, which corresponds to an overall convergence rate (in time) of approximately $2.17$. 
Overall, our proposed spatial discretization including the modified velocity space \eqref{V_space_mod_0} therefore behaves as expected with respect to accuracy and an energetically balanced structure.

\subsection{Tokamak Simulations and Solver Performance}

Having confirmed our spatial discretization's numerical properties, we next focus on solver efficiency by means of more realistic tokamak simulations based on equilibria from Grad-Shafranov solvers. For our finite element discretization, we consider 2D poloidal plane meshes as depicted in Figure \ref{fig_tokamak}, extruded toroidally -- in a higher order fashion to be further detailed below -- to obtain 3D hexahedral meshes. Projecting the Grad-Shafranov-solver based equilibria onto the finite element meshes introduces a non-negligible perturbation -- for further discussion in the context of curl- and div-conforming elements, see \cite{zhang2025structure} -- and in order to attenuate the resulting initial re-equilibration dynamics, we start our simulations using a time step of $\delta t = 0.02$ Alfvén times, which is then increased in each time step by a factor of $1.5$ until a final value of $\delta t = 1.5$ is reached (by approx. $\tau_A = 3.5$ Alfvén times). Finally, the tokamak domain contains a vacuum vessel and wall region, in which only the resistivity is evolved. As mentioned in Section \ref{sec_background}, throughout the numerical results section, we consider fixed Lundquist numbers $S_{vv} = 10^5$, $S_w = 10$ for these two regions.

\begin{figure}[ht]
\begin{center}
\includegraphics[width=1.0\textwidth]{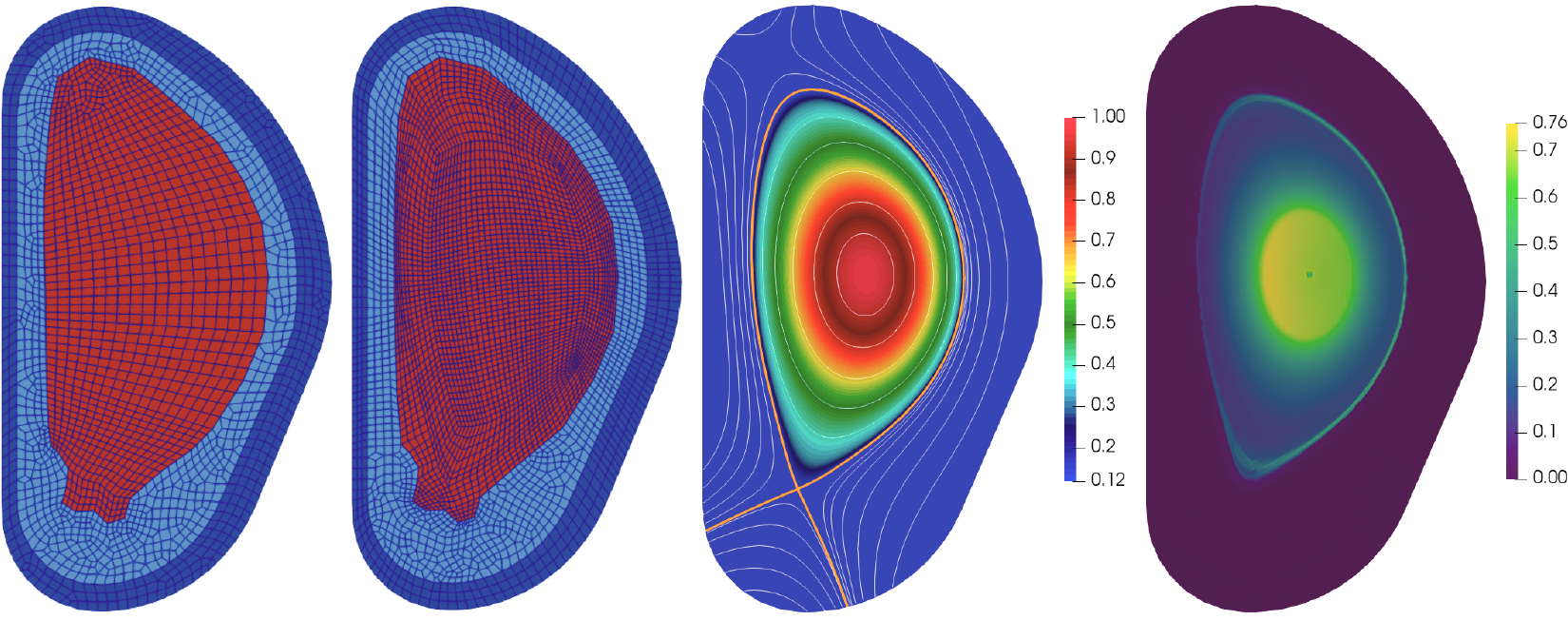}
\vspace{-5mm}
\caption{Poloidal base meshes and fields used for tokamak simulations. First, second images: general and separatrix aligned meshes, containing 2171 and 4528 elements, respectively. Blue, light blue, and red colors indicate the vacuum vessel, wall, and plasma regions, respectively. Third, fourth images: initial (non-dimensionalized) temperature field, including poloidal cross-section of magnetic field lines (with x-point contour in orange), and toroidal current density, respectively, for ITER-type equilibrium.} \label{fig_tokamak}
\end{center}
\end{figure}

\subsubsection{Solver robustness}  We will start with a low-dissipation equilibrium scenario, used to test solver robustness with respect to spatial resolution and time step size. Additionally, in tokamak scenarios, the (poloidal) $(r, z)$-plane and (toroidal) $\phi$-coordinate directions often differ strongly in resolution, leading to large mesh aspect ratios. This may affect solver performance too, and we will therefore also test with various toroidal resolutions.

We use a general, non-separatrix aligned 2D poloidal mesh (left image in Figure \ref{fig_tokamak}), of which we will consider two uniform spatial refinements, for a total of 3 spatial refinement levels $L_\mathrm{pol}$. The poloidal mesh is then extruded toroidally, and in the following, we will consider the following toroidal refinement levels of pairs (toroidal resolution, mesh degree), chosen to adequately resolve the tokamak's angular curvature: $L_\mathrm{tor} = \{0: (3, 6), 1: (6, 4), 2: (12, 2), 3: (24, 2)\}$. Finally, for the temporal refinement study, instead of capping at $\delta t = 1.5$, we cap at $\delta t \in \mathcal{T} = \{0.5, 0.75, 1, 1.25, 1.5, 2., 2.5, 3\}$. Altogether, the uniform and toroidal spatial as well as temporal refinement studies considered in this work are given by
\[
\begin{array}{c|c|c}
\text{Study} & (L_{\mathrm{pol}}\;|\;L_{\mathrm{tor}}) & \delta t_{max} \\
\hline
\text{uniform refinement} & \{(0\;|\;0), (1\;|\;1), (2\;|\;2)\} & 1.5 \\
\text{uniform refinement} & \{(0\;|\;1), (1\;|\;2), (2\;|\;3)\} & 1.5 \\
\text{toroidal refinement} & \{(0\;|\;0), (0\;|\;1), (0\;|\;2)\}, (0\;|\;3)\} & 1.5 \\
\text{varying } \delta t_{max} & (0 \;|\; 0) & \mathcal{T} \\
\text{varying } \delta t_{max} & (0 \;|\; 1) & \mathcal{T}
\end{array}
\]
Here, we note that the coarsest resulting 3D mesh -- given by poloidal/toroidal refinement levels $(0\;|\;0)$, that is no refinement poloidally, and a toroidal resolution of 3 layers -- contains $6513$ elements, while the finest one -- given by $(2\;|\;3)$ -- contains $2,500,992$ elements.

For our solver robustness tests, we consider a stable, axisymmetric equilibrium generated using \cite{liu2021parallel}, with a safety factor -- that is the ratio of toroidal to poloidal rotations of magnetic field lines -- varying between approximately $q_0=1.5$ at the magnetic axis and $q_{95}=6$ towards the separatrix (Figure \ref{fig_tokamak}). The equilibrium resembles an ITER discharge, with a magnetic field strength of $B_0 = 5.42$ T as well as a pressure of $p_0 = 656$ kPa, both measured at the magnetic axis. Once the magnetic and pressure fields are transferred to the MHD mesh, the magnetic field is divergence-cleaned (up to solver tolerance, which we set to $10^{-6}$). Further, for the pressure -- which by default is zero outside the separatrix in the Grad-Shafranov solver -- we first add a constant floor value and then non-dimensionalize by the resulting pressure at the magnetic axis. The non-dimensional pressure field then takes values in $[p_b, 1]$, and we set the floor such that $p_b = 0.01333$. With this non-dimensionalization, we have $\beta \approx 0.03$. As in the accuracy test, the initial density and temperature are set to (the projections of) $p_h^{0.3}$ and $p_h^{0.7}$, respectively. Here, we note that cross-interpolation from the Grad-Shafranov solution to the MHD mesh may locally lead to small negative pressure values. For this reason, we add smooth floor functions to our scheme at all occurrences of rational powers -- i.e., the computations of $n_h$, $T_h$ as functions of $p_h$ with a floor of $p_b$, as well as the $T_h$-dependent Lundqusit number \eqref{non_dim_coeff} with a floor of $T_b$. Furthermore, the dissipation coefficients are set to
\begin{equation}
    Re = 10^5, \;\;\;\; S = 10^7 (\mathbb{F}(T)/T_b)^{-3/2}, \;\;\;\; Pe_\parallel = 20, \;\;\;\; Pe_\perp = 10^7,
\end{equation}
where $\mathbb{F}(T)$ within $S$ denotes the aforementioned floor function. In terms of the system's stiffness, we additionally note that the Alfvén and sound wave CFL numbers are defined by
\begin{equation}
    \text{CFL}_A = \frac{k \delta t}{\sqrt{|n_h|}} \left(\frac{|\mathbf{B}_{p, h}|}{h_p} + \frac{|B_{t,h}|}{h_t}\right), \;\;\;\;\;\; \text{CFL}_s = \frac{k \delta t \sqrt{\gamma \beta |T_h|}}{h_p},
\end{equation}
for polynomial degree $k$, poloidal, toroidal components $\mathbf{B}_{p, h}$, $B_{t,h}$ of $\mathbf{B}_h$, and poloidal, toroidal local cell length scales $h_p$, $h_t$, respectively. Note that in CFL$_s$, we use $h_p$ since sound waves are isotropic, and for our anisotropic mesh, $h_p < h_t$. For the non-separatrix aligned mesh at the lowest refinement level with 3 toroidal layers, together with the above field setup at a time step $\delta t = 1.5$, we find CFL$_A \approx 40$, CFL$_s \approx 6.5$. Further, we note that for our coarsest mesh $(0\;|\;0)$ (and its spatially uniform refinements), the toroidal-to-poloidal mesh aspect ratio within the plasma region varies approximately from 34 to 62.

The simulations are run up to $\tau_A = 20$ Alfvén times, and inner solver counts for the Schur complement in $\mathbf{B}_h$ are shown in Tables \ref{Table_SchurB_robustness_space} and \ref{Table_SchurB_robustness_time}. For uniform spatial refinement, we find an increase in iterations by a range of factors between 1.26 and 1.82. We expect achieving the ideal scaling factor of $1$ to be difficult given the complexity of the mixed-form anisotropic curl-curl operator underlying the Alfvén wave dynamics, the realistic scenario with a non-regular mesh and strongly varying prognostic fields, and also the strongly anisotropic mesh configuration. Further, for the highest resolution run with the largest increase by a factor of 1.82, considerations related to MPI decomposition may additionally play a role; see Remark \ref{remark_weak_scaling} below. In line with the mesh anisotropy, we observe that for the case of refining in the toroidal direction only, the iteration count decreases as the toroidal resolution is increased, until the last refinement. This decreasing, then increasing change in factors may be due to a balance of decreasing mesh anisotropy versus increasing Alfvén wave CFL due the magnetic fields' large toroidal component $B_{t,h}$, as the toroidal resolution is refined. Finally, for the varying time step study, we also observe a moderate increase in inner iterations, with a factor of about 1.4 and 3.7 (relative to $\delta t_{\max} = 0.5$) when moving from $\delta t = 0.5$ to $\delta t = 1$ and to $\delta t = 3$, respectively. Overall, given the challenging nature of this problem -- in which we specifically target a realistic tokamak problem setup -- we find these increases for spatial and temporal refinements to be reasonable.

Before moving on to our second quantity of interest -- wall-clock times for longer tokamak simulations -- we note that for all of the tests, we generally observed 2 to 3 nonlinear iterations per time step. Further, we do not show iteration counts for the first IMIM stage's solves related to sound waves and magnetic field stabilization, nor the heat flux solve, since these are not the focus of this work. In practice, our simple boomerAMG-based preconditioning strategy for the sound waves and heat flux will generally not be robust. In particular, robust solvers for the heat flux, especially for realistic anisotropy ratios, is an ongoing research topic.

\begin{table}
\begin{center}
\begin{tabular}{|c|cc|cc|cc|}
\hline
& \multicolumn{2}{c|}{Uniform ref., 3 tor. layers}
& \multicolumn{2}{c|}{Uniform ref., 6 tor. layers}
& \multicolumn{2}{c|}{Tor. ref.} \\
\cline{2-7}
Level
& Schur-B its & Rel. factor
& Schur-B its & Rel. factor
& Schur-B its & Rel. factor \\
\hline
0 & 48.00 &  & 43.75 &  & 48.88 &  \\
1 & 67.62 & 1.41 & 55.00 & 1.26 & 45.25 & 0.93 \\
2 & 98.50 & 1.46 & 100.5 & 1.83 & 27.50 & 0.61 \\
3 &&&&& 31.5 & 1.16\\
\hline
\end{tabular}
\caption{Schur-B iteration counts per nonlinear iteration averaged over the last four time steps for uniform spatial refinement tests (first two columns) as well as for toroidal refinement test (third column).}
\label{Table_SchurB_robustness_space}
\end{center}
\end{table}

\begin{table}
\begin{center}
\begin{tabular}{cc}
\begin{tabular}{|c|c|c|}
 \hline
 \multicolumn{3}{|c|}{3 toroidal layers} \\
 \hline
 $\delta t_{\max}$ & Schur-B its. & Factor wrt. $\delta t_{\max} \!=\! 0.5$ \\
 \hline
 0.50 & 18.50 &  \\
 0.75 & 23.50 & 1.25 \\
 1.00 & 26.75 & 1.42 \\
 1.25 & 34.25 & 1.82 \\
 1.50 & 49.62 & 2.64 \\
 2.00 & 55.00 & 2.93 \\
 2.50 & 68.12 & 3.62 \\
 3.00 & 69.50 & 3.70 \\
 \hline
\end{tabular}
&
\begin{tabular}{|c|c|c|} 
 \hline
 \multicolumn{3}{|c|}{6 toroidal layers} \\
 \hline
 $\delta t_{\max}$ & Schur-B its. & Factor wrt. $\delta t_{\max} \!=\! 0.5$ \\
 \hline
 0.50 & 18.25 &  \\
 0.75 & 19.75 & 1.08 \\
 1.00 & 26.5 & 1.45 \\
 1.25 & 29.5 & 1.61 \\
 1.50 & 40.62 & 2.22 \\
 2.00 & 47.75 & 2.61 \\
 2.50 & 53.75 & 2.94 \\
 3.00 & 56.12 & 3.07 \\
 \hline
\end{tabular}
\end{tabular}
\caption{Schur-B iteration counts per nonlinear iteration averaged over the last four time steps for the temporal refinement test, using 3 (left) and 6 (right) toroidal layers.}
\label{Table_SchurB_robustness_time}
\end{center}
\end{table}

\paragraph{Solver strategy components.} Other than the IMIM-split time stepping method, two key components of our solver strategy relate to using AMS for the Schur complement in $\mathbf{B}_h$, as well as considering the velocity in the modified function space \eqref{V_space_mod_0}. Alternatively, one could e.g., use boomerAMG for the Schur complement, and simply consider the velocity to be a function of the edge-based space $\mathring{N}c_k^e(\Omega)$. To test this, we consider our base poloidal mesh with two toroidal refinements. Replacing AMS by boomerAMG, we find that as the time step is increased, at $\delta t = 0.513$, we require over 4000 inner iterations. Keeping AMS, but replacing the modified velocity space with $\mathring{N}c_k^e(\Omega)$, we find that at $\delta t = 0.342$ and $\delta t = 0.513$, the non-modified velocity space setup takes 41 and 70 inner iterations, respectively, while the modified setup takes 17 and 21, respectively (growing to only 27.5 at $\delta t = 1.5$; cf. the right-most table in Tables \ref{Table_SchurB_robustness_space}). This indicates that while our strategy including the modified velocity space remains relatively robust -- with respect to temporal refinement -- this is not the case when using $\mathring{N}c_k^e(\Omega)$ instead. A similar observation can be shown to also hold true for spatial refinement.

\begin{remark}[Weak scaling] \label{remark_weak_scaling}
The uniform spatial refinement study starting with 3 toroidal layers was performed using 16, 128, and 1024 MPI processes, respectively, on the NERSC HPC cluster Perlmutter. For the three refinements, the magnetic field Schur complement solve takes approximately $6.2$, $10$, and $30.6$ seconds per nonlinear iteration, respectively. The first refinement step shows a moderate increase in runtime by a factor of 1.6. The final refinement step exhibits an additional increase by a factor of 3.0, which we suspect is related to computational overhead from increased cross-node communication (noting that each node can host up to 64 MPI processes), as well as to domain-decomposition-related effects in Firedrake, where MPI subdomains consist of toroidally extruded subdivisions of the underlying poloidal 2D mesh. In particular, this increases the fraction of degrees of freedom contained in MPI halo regions within the poloidal plane, as the local MPI subdomains become progressively thinner under refinement. Similarly, the uniform spatial refinement study starting with 6 toroidal layers was performed using 32, 256, and 2048 MPI processes, respectively, and we suspect that the increase in MPI ranks also contributes to the larger increase in inner iterations (from 55 to 100.5; cf. the center table in Table \ref{Table_SchurB_robustness_space}).
\end{remark}

\subsubsection{Long tokamak runs} \label{sec_tokamak_long_runs}
In our final suites of numerical tests, we consider long tokamak runs to examine our numerical scheme's wall-clock times as well as accuracy for more practical scenarios. For this purpose, we consider the separatrix-aligned mesh (second image in Figure \ref{fig_tokamak}), together with a background pressure of $p_b = 0.05$ to avoid instabilities related to negative temperature or density values -- these could alternatively be avoided e.g., by using limiters, which for simplicity we do not consider in this work. For this choice of $p_b$, we note that the floor function $\mathbb{F}$ used in the robustness studies above is not needed.

We consider three separate configurations, given by a near-equilibrium configuration, a vertical displacement event (VDE), and a 1-1 kink mode:
\[
\begin{array}{l|l|c|c|c|c|c|c}
\text{Case}
& \text{IC}
& t_{\max}
& Re
& S
& Pe_{\parallel}
& Pe_{\perp}
& \text{(Tor. layers, mesh deg.)}
\\
\hline
\text{Near equilibrium} & \text{ITER} & 1000 & 10^5& 10^7& 20& 10^7& (3, 6)\\
\text{VDE}              & \text{ITER} & 6000 & 10^4& 10^4& 20& 2\times 10^3& (3,6)\\
\text{Kink}             & q < 1 & 6000 & 10^5& 10^4& 100& 10^5& (6, 4) \\
\end{array}
\]
where the ``ITER'' initial condition corresponds to the one used in the robustness study, while the ``$q<1$'' one corresponds to another equilibrium in which the safety factor is less than one in a small region off the magnetic axis, allowing for kink instabilities. Further, the kink instability scenario includes a 1-1 mode perturbation $p_{\mathrm{pert}}$ of magnitude up to $10^{-3}$, which is centered at the region where $q < 1$, and which is added to the initial pressure profile. Note that in order to resolve the 1-1 mode toroidally, we consider 6 toroidal layers, which for our second order spaces corresponds to 12 degrees of freedom in the toroidal direction. Finally, each simulation is run with a maximum time step of $\delta t_{\max} = 1.5$, using 256 MPI processes distributed across 4 nodes on the NERSC HPC cluster Perlmutter.

Averages solver times for the full scheme's IMIM components as well as $\mathbf{B}_h$-Schur complement iteration counts are given in Table \eqref{Table_long_run_solver_stats}. We find that overall, the VDE and near-equilibrium runs have similar iteration counts and solve times, except for a slight increase for the latter in the near-equilibrium run due to a more challenging heat flux and sound wave solve, both of which are likely related to the less dissipative regime in the near-equilibrium run. Further, the kink mode simulation requires significantly more iterations due to the more challenging 3D perturbed dynamics, which -- together with the higher toroidal resolution -- leads to overall higher solver times. However, the runtime generally remains moderate even for the kink simulation, with an overall runtime of approximately 2:40 hours for the first 1000 Alfven times of this challenging 3D tokamak simulation. This is due to a combination of the moderate inner solver iteration and consistently low nonlinear iteration counts (2 to 3) at a time step of $\delta t = 1.5 \tau_A$, together with a toroidal resolution that is not unnecessarily high, owing to the solver's reasonable robustness with respect to mesh anisotropy and the fourth-order mesh in the toroidal direction. Additionally, while the table only considers the window $\tau_A \in [850-1000]$, we note that we found similar values for the VDE and kink mode simulations beyond $\tau_A > 1000$. Finally, averaging over the last 500 time steps (instead of 100) yields comparable mean iteration counts for the equilibrium and VDE cases. For the nonlinear kink simulation, larger fluctuations are observed for this larger averaging window (e.g., 58.7 $\pm$ 12.4 Schur-B iterations), reflecting the increased complexity of the 3D dynamical state.

\begin{table}
\begin{center}
\begin{tabular}{|c|cc|cc|cc|}
\hline
&
\multicolumn{2}{c|}{Alfv\'en wave + heat flux}
&
\multicolumn{2}{c|}{Sound wave}
&
\multicolumn{2}{c|}{B stabilization}
\\
\cline{2-7}
Run
& $S'_{B,n_0}$ Its. & Time [s]
& $S'_{B,n_0}$ Its. & Time [s]
& $S'_{B,n_0}$ Its. & Time [s]
\\
\hline
Equilibrium
& $41.7 \pm 1.4$
& $2.44 \pm 0.08$
& $74.3 \pm 1.6$
& $1.93 \pm 0.16$
& $89.2 \pm 0.4$
& $0.293 \pm 0.007$
\\
VDE
& $49.0 \pm 0.0$
& $2.27 \pm 0.02$
& $50.0 \pm 0.3$
& $1.62 \pm 0.02$
& $91.3 \pm 0.9$
& $0.338 \pm 0.007$
\\
Kink
& $67.7 \pm 0.2$
& $4.84 \pm 0.04$
& $132.9 \pm 9.5$
& $4.70 \pm 0.28$
& $108.4 \pm 1.4$
& $0.887 \pm 0.023$
\\
\hline
\end{tabular}
\caption{Average Schur-B iteration counts per nonlinear iteration and overall solver times per time step over the 100 time step interval $\tau_A \in [850-1000]$ Alfvén times. Values are reported as mean $\pm$ standard deviation.}
\label{Table_long_run_solver_stats}
\end{center}
\end{table}

Next to solver iteration counts and wall-clock times, we also consider the simulations' field development. Figure \ref{fig_VDE_kink} depicts a poloidal cross-section of the kink and VDE simulations, and Figure \ref{fig_equilibrium_long} depicts the MHD equilibrium and relative total energy departures over time for the near-equilibrium run. We find that as expected, the initial seed perturbation $p_\mathrm{pert}$ leads to a large-scale, 3D instability in the kink mode simulation. Further, we find that this evolution continues up until our maximum time $t_{\max} = 6000$, with the overall temperature decreasing over time due to breaking flux surfaces and the resulting parallel heat conduction towards the plasma-facing wall. For corresponding figures, see Appendix \ref{app_figures_MHD}. Similarly, as expected, in the VDE simulation the plasma column moves upward over time and touches the upper plasma-facing wall.

Further, for the near-equilibrium run, we find that after an initial phase with stronger dynamics -- recalling that the Grad-Shafranov solution to MHD initial condition transfer leads to a small initial equilibrium error -- the state slowly relaxes back to an equilibrium state. Additionally, as expected, since the resistivity and perpendicular diffusion take effect at a much slower time scale, the relative energy error remains below the solver tolerance of $10^{-5}$. 

\begin{figure}[ht]
\begin{center}
\includegraphics[width=1.0\textwidth]{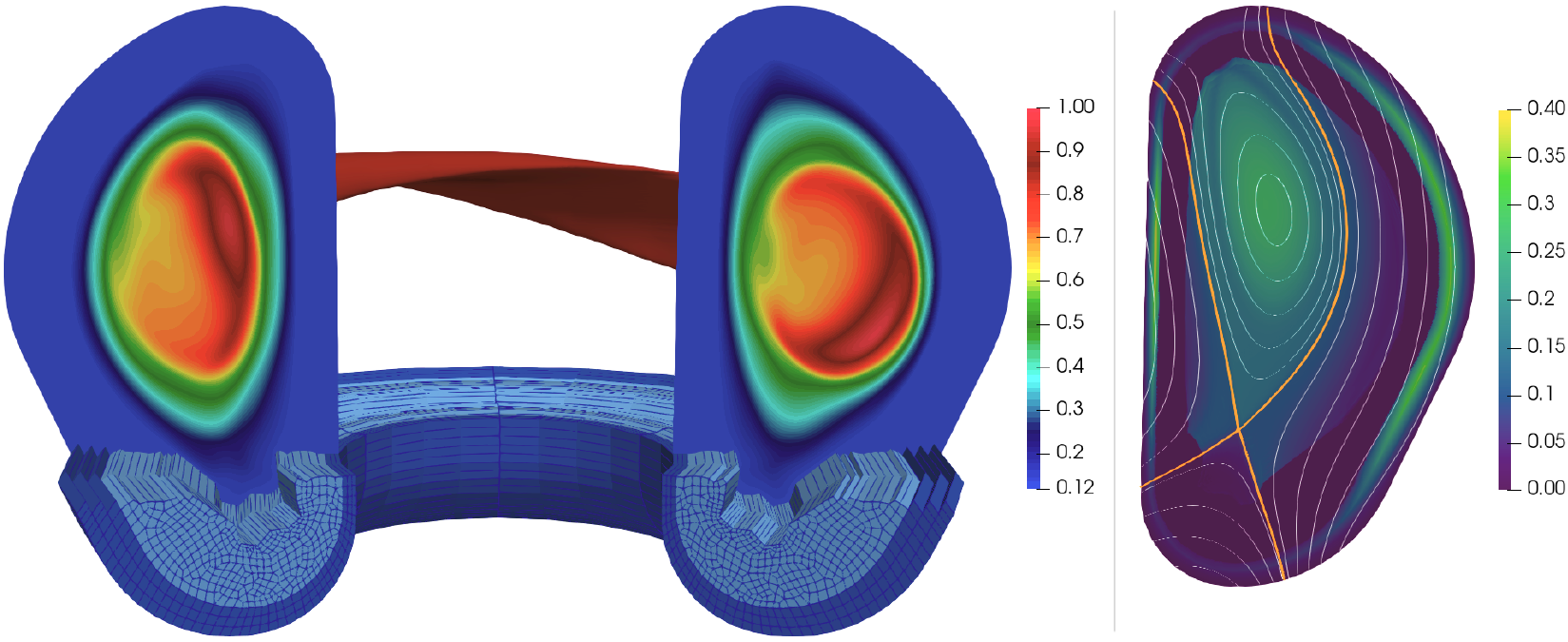}
\vspace{-5mm}
\caption{Left image: poloidal cross-section of temperature profile for kink mode run at $\tau_A = 1000$ Alfvén times; portion of toroidally $4^{th}$ order mesh for $\Omega_v$, $\Omega_w$ near the bottom; 3D contour of temperature field evaluated at $T_h = 0.9$ near the center. Right image: poloidal cross-section of toroidal current density component, as well as magnetic field lines (with x-point contour in orange), for VDE run at $\tau_A = 6000$.} \label{fig_VDE_kink}
\end{center}
\end{figure}

\begin{figure}[ht]
\begin{center}
\begin{minipage}{0.49\textwidth}
\centering
\includegraphics[width=\textwidth]{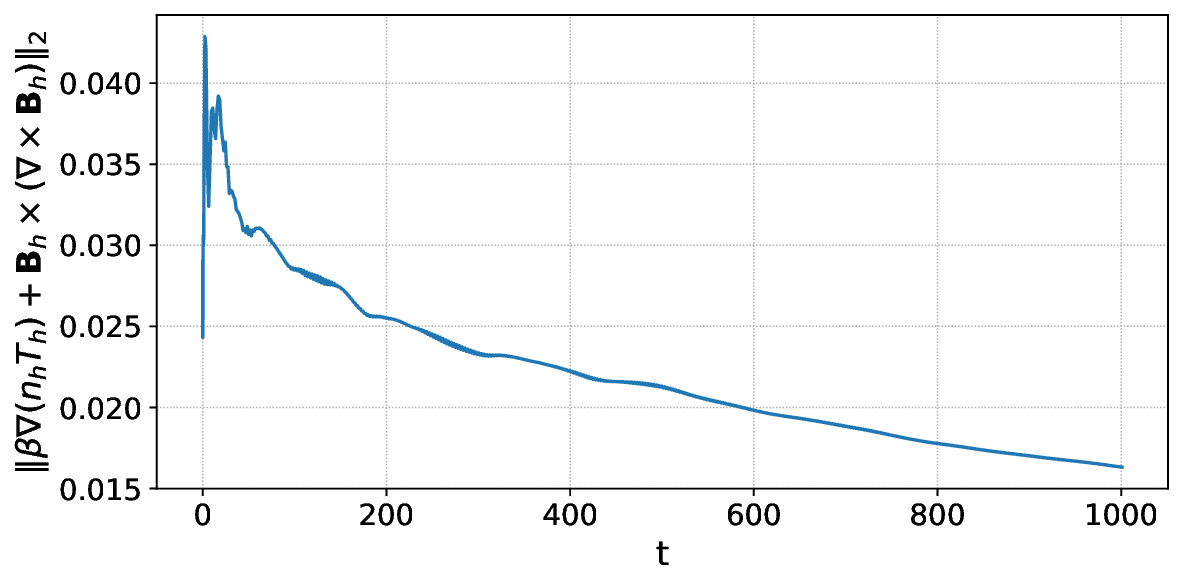}
\end{minipage}
\hfill
\begin{minipage}{0.49\textwidth}
\centering
\includegraphics[width=\textwidth]{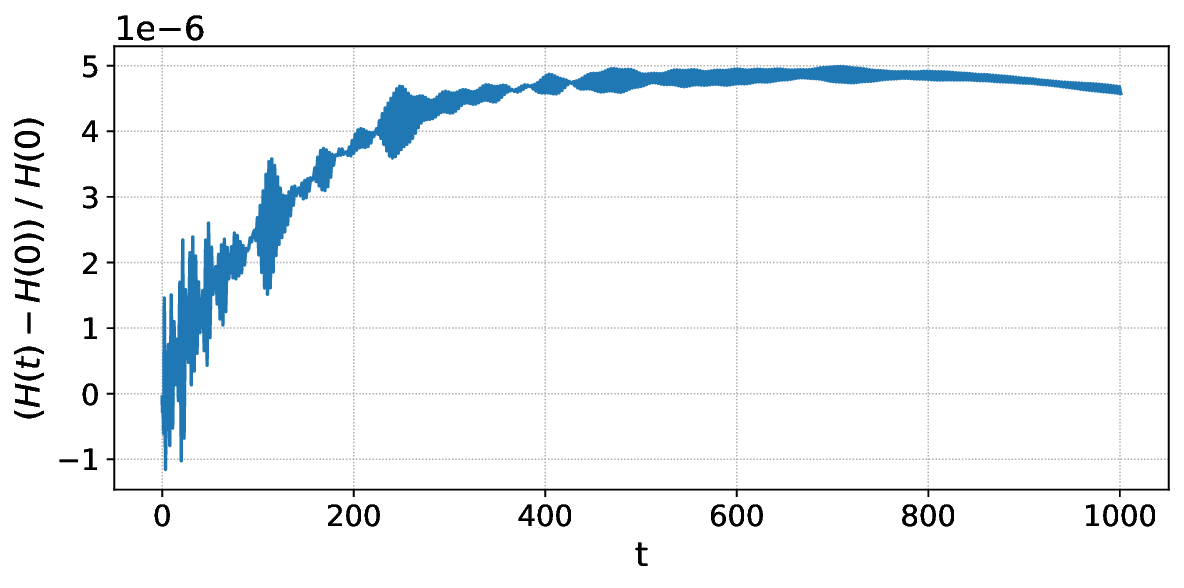}
\end{minipage}
\vspace{-3mm}
\caption{Diagnostics for the near-equilibrium tokamak simulation. Left: departure from an MHD equilibrium state measured in the $L^2$ norm. Right: relative total energy departure over time.}
\label{fig_equilibrium_long}
\end{center}
\end{figure}
%
%
\section{Conclusion} \label{sec_conclusion}
In this work, we presented a novel solver strategy for the MHD equations in the low-$\beta$ regime of magnetic confinement fusion, targeting the system's stiff magnetic wave coupling. The approach combines a curl-conforming finite element discretization of the magnetic field $\mathbf{B}$ -- which also preserves a discrete divergence-free property -- with a physics-based block preconditioner centered on a Schur complement in $\mathbf{B}$. To obtain an effective Schur complement approximation, we employ an implicit-implicit split discretization that removes sound-wave and velocity-dissipation contributions from the Schur complement. In addition, we construct a specialized velocity finite element space that incorporates the cross product with $\mathbf{B}$ and is designed such that the velocity-space projection vanishes in the approximate Schur complement. Finally, the resulting Schur complement operator is solved using an auxiliary-space Maxwell (AMS) preconditioner.

In numerical tests, we confirmed the accuracy of the spatial discretization and its energetically consistent structural properties. Further, we demonstrated moderate robustness of the solver strategy with respect to spatial and temporal refinement for a realistic ITER-like test case featuring general reactor geometry, a low-dissipation regime, and highly anisotropic meshes. In particular, we showed that in addition to using AMS for the Schur complement, the custom velocity space is essential in our approach to ensure low iteration counts at larger time steps and finer spatial resolutions. Finally, we conducted longer simulations for three different tokamak scenarios, including a challenging kink-mode instability. In all cases, the observed field evolution was as expected, while the solver maintained relatively low iteration counts and practical wall-clock times throughout. Altogether, this work provides an efficient and accurate solver framework for MHD simulations in realistic 3D tokamak geometries.

In future work, we plan to further increase our scheme's efficiency through incorporating known, more robust solvers for the anisotropic heat flux and sound wave dynamics. Additionally, we aim to incorporate more advanced physical terms in our scheme, including e.g., more suitable boundary and interface conditions, as well as extended MHD terms such as Hall MHD and kinetic Alfvén waves. In particular, our solver strategy based on a Schur complement in the magnetic field equation is readily amenable to such terms, since they relate to a generalized Ohm's law applied in the magnetic field equation.

\textbf{Acknowledgments}\\
This work was supported by the Laboratory Directed Research and Development program of Los Alamos National Laboratory, under project number 20240261ER, as well as the U.S. Department of Energy Office of Fusion Energy Sciences Base Theory Program, at Los Alamos National Laboratory under contract No. 89233218CNA000001. The computations have been performed using resources of the National Energy Research Scientific Computing Center (NERSC), a U.S. Department of Energy Office of Science User Facility operated under Contract No. DE-AC02-05CH11231. Further, we would like to thank Oliver Krzysik for his many helpful remarks. Los Alamos National Laboratory report number LA-UR-26-27307.
%
%
%
%
\appendix

\section{Mixed finite element stability for Lorentz force term} \label{app_2D_MHD}
To further study the choice of function space for the discrete velocity $\mathbf{V}_h$ given the Lorentz force term, we consider the case of tokamak domains and axisymmetric fields. In this case, the problem is reduced to a 2D one in the $(r,z)$-plane for cylindrical coordinates $(r, \phi, z)$. The velocity and magnetic vector fields can then be decomposed as
\begin{align}
&\mathbf{V} = (V_r, V_\phi, V_z)(r, \phi, z, t) \;\;\rightarrow\;\; (V_r, V_\phi, V_z)(r, z, t) = (V_r, V_z) + V_\phi \mathbf{e}_\phi, \\
&\mathbf{B} = (B_r, B_\phi, B_z)(r, \phi, z, t) \;\rightarrow\; (B_r, B_\phi, B_z)(r, z, t) = (B_r, B_z) + B_\phi \mathbf{e}_\phi,
\end{align}
for angular unit vector $\mathbf{e}_\phi$, and where $\mathbf{B}_p = (B_r, B_z)$, $\mathbf{V}_p = (V_r, V_z)$ are 2D vector fields, while $V_\phi$ and $B_\phi$ are scalar fields corresponding to out-of-plane vector components. The 3D $\mathbf{B}$-$\mathbf{V}$ coupling via the Lorentz force can be decomposed into three sets of couplings that relate the cylindrical $(r,z)$- and $\phi$-components of $\mathbf{B}$-$\mathbf{V}$ to each other. In particular, the $\mathbf{V}_p$-$B_\phi$ coupling can be shown to read
\begingroup
\addtolength{\jot}{2mm}
\begin{subequations}
\begin{align}
&\left\langle r n \mathbf{v} , \pp{\mathbf{V}_p}{t} \right\rangle + \left\langle\mathbf{v}, B_t \nabla (rB_t) \right\rangle = 0 &\forall \mathbf{v} \in \mathbb{V}_{V_p}, \\
& \left\langle r \chi, \pp{B_t}{t}\right\rangle - \left\langle \nabla (r \chi), B_t \mathbf{V}_p \right\rangle = 0 & \forall \chi \in \mathbb{V}_{B_t}, \label{2D_MHD_Bt}
\end{align}
\end{subequations}
\endgroup
for $L^2$-inner product $\langle.,.\rangle$, and where $\nabla = (\partial_r, \partial_z)$ and $\mathbb{V}_{V_p}$, $\mathbb{V}_{B_t}$ are suitable discrete function spaces. For 3D tokamak meshes in which each edge lies either in an $(r,z)$-plane or is parallel to $\mathbf{e}_\phi$ -- as is the case for the extruded meshes considered in this work -- the discrete space $Nc_k^e(\Omega)$ admits a tensor-product structure. In particular, it contains a 2D CG space $Q_k$ associated with the $\phi$-component of the vector fields in $Nc_k^e(\Omega)$ (see, e.g., \cite{mcrae2016automated} for details on tensor-product spaces). In other words, a consistent choice in the reduced 2D equation \eqref{2D_MHD_Bt} is to set $\mathbb{V}_{B_t} = Q_k$ for the scalar field $B_t$ corresponding to the 2D out-of-plane $\phi$-component of $\mathbf{B}$. Mixed finite element stability theory then suggests to set $\mathbb{V}_{V_t}$ to a curl-conforming finite element space, and by an analogous tensor-product argument, this in turn suggests setting the discrete 3D velocity space to $Nc_k^e$. Finally, we note that similar arguments are less clear for the two other couplings related to $\mathbf{V}_t$-$\mathbf{B}_p$ and $\mathbf{V}_p$-$\mathbf{B}_p$, as well as additional couplings in the full 3D case. In practice, the corresponding numerical modes then need to be taken care of using sufficiently strong physical diffusion or numerical stabilization terms.
\section{Structural properties of the spatial discretization} \label{app_structure}
Here, we briefly discuss some of the structural properties of the spatial discretization \eqref{3D_MHD_discr}. First, the discrete magnetic field equation \eqref{B_eqn_discr} conserves the discrete weak divergence $\delta_B \in Q_k(\Omega)$, defined according to
\begin{align}
&\left\langle \chi, \delta_B \right\rangle_\Omega = - \left\langle \nabla \chi, \mathbf{B} \right\rangle_\Omega + \int_{\partial\Omega} \chi g_1 \; dS & \forall \chi \in Q_k(\Omega), \label{def_weak_B_div}
\end{align}
where $g_1$ is defined along the domain's boundary $\partial \Omega$ as $\mathbf{B}_h \cdot \mathbf{n}|_{t=0}$. This can be shown by taking the time derivative of \eqref{def_weak_B_div} and noting that in the context of finite element exterior calculus, for $\chi \in Q_k(\Omega)$, we have $\nabla \chi \in Nc_k^e(\Omega)$. We can then substitute for the magnetic field equation \eqref{B_eqn_discr} by setting $\mathbf{\Sigma} = \nabla \chi$, and the result $\delta_B \equiv 0$ then follows using the discrete vector calculus identity $\nabla \times \nabla \chi \equiv 0$; for details, see \cite{wimmer2024structure}.

In periodic domains, the semi-discrete system \eqref{3D_MHD_discr} can further be shown to be consistent with respect to the discrete total energy
\begin{equation}
    H(n_h, \mathbf{V}_h, T_h, \mathbf{B}_h) = \int_\Omega \left( \frac{1}{2}n_h|\mathbf{V}_h|^2 + \frac{\beta n_h T_h}{\gamma - 1} + \frac{1}{2}|\mathbf{B}_h|^2 \right)\text{dx}.
\end{equation}
For this purpose, an analogous argument as used in our previous work \cite{wimmer2024structure} can be employed, using the chain rule in the time derivative and substituting for the discrete evolution equations \eqref{3D_MHD_discr}. Setting the test functions to $\chi = \beta T_h/(\gamma-1) + P_h \in Q_k$, $\mathbf{w} = \mathbf{V}_h \in Nc_k^e$\footnote{Note that this substitution works independently of the choice of discrete function space for $\mathbf{V}_h$, including the modified space \eqref{V_space_mod_0}.}, $\eta = \beta \in Q_k$, $\phi_h = Pe_\Delta \zeta_h \in dQ_{k-1}$, and $\mathbf{\Sigma}_h = \mathbf{B}_h \in Nc_k^e$, we find
\begingroup
\addtolength{\jot}{2mm}
\begin{align}
    \frac{dH}{dt} =&\left\langle \pp{H}{n_h}, \pp{n_h}{t} \right\rangle + \left\langle \dd{H}{\mathbf{V}_h}, \pp{\mathbf{V}_h}{t} \right\rangle + \left\langle \dd{H}{T_h}, \pp{T_h}{t} \right\rangle + \left\langle \dd{H}{\mathbf{B}_h}, \pp{\mathbf{B}_h}{t} \right\rangle \nonumber \\
    =&\left\langle \frac{\beta T_h}{\gamma - 1} + P_h , \pp{n_h}{t} \right\rangle + \left\langle n_h \mathbf{V}_h, \pp{\mathbf{V}_h}{t} \right\rangle + \left\langle \frac{n_h \beta}{\gamma - 1}, \pp{T_h}{t} \right\rangle + \left\langle \mathbf{B}_h, \pp{\mathbf{B}_h}{t} \right\rangle \label{H_subst} \\
    = &- \left\|\sqrt{Re^{-1}} \nabla \mathbf{V}\right\|_2^2 - \left\|\sqrt{S^{-1}} \nabla \times \mathbf{B}\right\|_2^2 \nonumber \\
    &- L_{\kappa_n}\left(\frac{\beta T_h}{\gamma - 1} + P_h, n_h\right)- L_{\kappa_V}(f(\mathbf{V}), f(\mathbf{V})) - L_{\kappa_B}(\mathbf{B}, \mathbf{B}), \label{H_final}
\end{align}
\endgroup
where going from \eqref{H_subst} to \eqref{H_final}, coupling terms between the evolution equations cancel and the curl-based velocity advection term vanishes; for details see \cite{wimmer2024structure}. Further, the temperature equation's interior penalty term vanishes since it contains a gradient of the test function, and we set $\eta = \beta$, which is a constant. Finally, the first two terms in \eqref{H_final} correspond to the physical dissipative viscous and resistive terms. Further, the interior penalty terms $L_{\kappa_V}$ and $L_{\kappa_B}$ are symmetric, and can be seen as acting as grid-scale kinetic and magnetic energy dissipation. Finally, the term
\begin{equation}
    - L_{\kappa_n}\left(\frac{\beta T_h}{\gamma - 1} + P_h, n_h\right),
\end{equation}
which arises from density stabilization (together with a weak boundary penalty in the case of non-periodic domains), is indefinite. As a consequence, the spatial discretization is energetically consistent only up to this contribution. On the other hand, in the absence of boundaries, it is still consistent with respect to global mass conservation. To see this, we note that $L_{\kappa_n}$ only contains a gradient of the test function in the interior penalty term, and therefore vanishes for constant test functions. Mass conservation then follows by choosing $\chi = 1$ in the continuity equation.
\section{Additional tokamak MHD simulation figures} \label{app_figures_MHD}
\begin{figure}[ht]
\begin{center}
\includegraphics[width=1.0\textwidth]{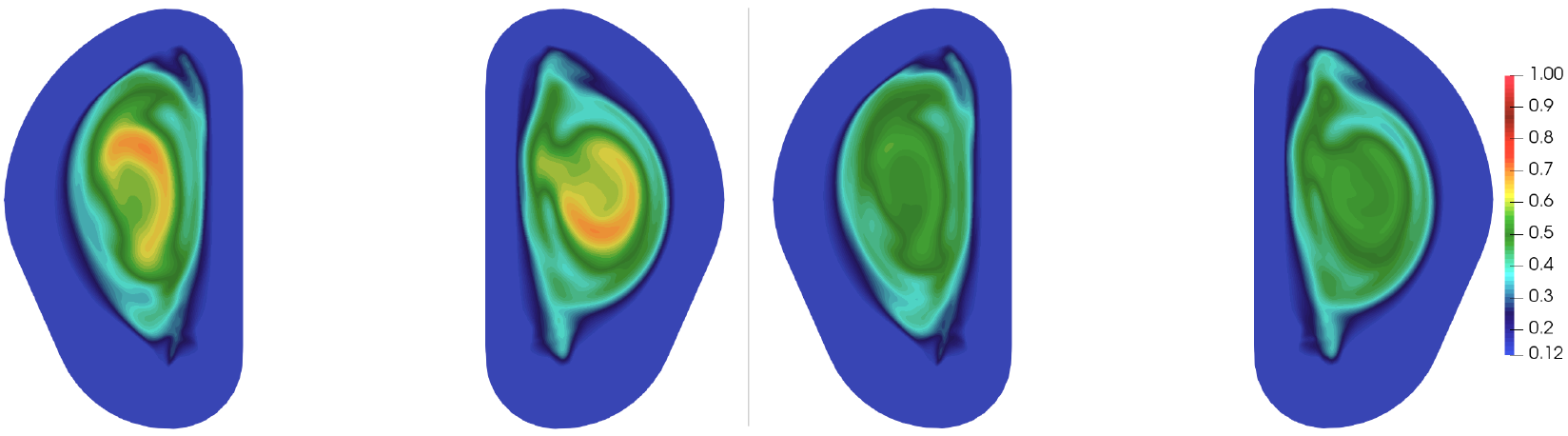}
\vspace{-5mm}
\caption{Additional images for kink run in Section \ref{sec_tokamak_long_runs}, depicting the temperature field at different times. Left: $\tau_A = 5000$. Right: $\tau_A = 6000$.}
\end{center}
\end{figure}
%
%
%
%
\bibliographystyle{abbrv}
\bibliography{paper_3D_MHD_solver.bib}
\end{document}